\documentclass[10pt]{amsart}
      \usepackage[mathscr]{eucal}
      \usepackage{amsmath,amsfonts}
      \parskip\smallskipamount
          \newtheorem{theorem}{Theorem}[section]
      
      \newtheorem{proposition}[theorem]{Proposition}
      \newtheorem{corollary}[theorem]{Corollary}
      \newtheorem{lemma}[theorem]{Lemma}

      \makeatletter
      \@addtoreset{equation}{section}
      \makeatother

      \newcommand{\CC}{{\mathbb C}}
      \newcommand{\NN}{{\mathbb N}}
      
      \newcommand{\ZZ}{{\mathbb Z}}
      \newcommand{\DD}{{\mathbb D}}
      
      \newcommand{\FF}{{\mathbb F}}
      \newcommand{\TT}{{\mathbb T}}

 \newcommand{\GG}{{\mathbb G}}

\newcommand{\HH}{{\mathbb H}}
\newcommand{\KK}{{\mathbb K}}

      \newcommand{\cA}{{\mathcal A}}

      \newcommand{\cD}{{\mathcal D}}
      \newcommand{\cE}{{\mathcal E}}
      
      \newcommand{\cG}{{\mathcal G}}
      \newcommand{\cH}{{\mathcal H}}
      
      \newcommand{\cK}{{\mathcal K}}
      \newcommand{\cL}{{\mathcal L}}
      \newcommand{\cM}{{\mathcal M}}
       \newcommand{\cO}{{\mathcal O}}

      \newcommand{\cP}{{\mathcal P}}

      \newcommand{\supp}{\hbox{\rm{supp}}\,}

      \newdimen\expt
      \def\boxit#1{\setbox0\hbox{$\displaystyle{#1}$}
            \hbox{\lower.4\expt
       \hbox{\lower3\expt\hbox{\lower\dp0
            \hbox{\vbox{\hrule height.4\expt
       \hbox{\vrule width.4\expt\hskip3\expt
            \vbox{\vskip3\expt\box0\vskip2\expt}%
       \hskip3\expt\vrule width.4\expt}\hrule height.4\expt}}}}}}
      
\begin{document}
       \pagestyle{myheadings}
      \markboth{ Gelu Popescu}{ Functional Calculus and Multi-Analytic Models on regular $\Lambda$-polyballs }

      \title [  Functional Calculus and Multi-Analytic Models  ]
      { Functional Calculus and Multi-Analytic Models on regular $\Lambda$-polyballs
      }
        \author{Gelu Popescu}
\date{January 28, 2020}
     \thanks{Research supported in part by  NSF grant DMS 1500922}
      \address{Department of Mathematics, The University of Texas
      at San Antonio \\ San Antonio, TX 78249, USA}
      \email{\tt gelu.popescu@utsa.edu}

\begin{abstract}
In a recent paper, we introduced the standard $k$-tuple ${\bf S}:=({\bf S}_1,\ldots, {\bf S}_k)$ of pure row isometries ${\bf S}_i:=[S_{i,1}\cdots  S_{i,n_i}]$ acting on the Hilbert space $\ell^2(\FF_{n_1}^+\times \cdots \times \FF_{n_k}^+)$,  where $\FF_n^+$ is the unital free semigroup with $n$ generators,  and showed that ${\bf S}$ is the universal  $k$-tuple of doubly $\Lambda$-commuting row isometries,  i.e.
$$ S_{i,s}^* S_{j,t}=\overline{\lambda_{ij}(s,t)}S_{j,t}S_{i,s}^*
$$
for every $i,j\in \{1,\ldots, k\}$ with $i\neq j$ and every  $s\in \{1,\ldots, n_i\}$, $t\in \{1,\ldots, n_j\}$,
where 
$\Lambda_{ij}:=[\lambda_{i,j}(s,t)]$ is an $n_i\times n_j$-matrix  with the entries in   $\TT:=\{z\in \CC: \  |z|=1\}$ and   $\Lambda_{j,i}=\Lambda_{i,j}^*$. It was also proved that  the set of all $k$-tuples $T:=(T_1,\ldots, T_k)$ of row operators $T_i:=[T_{i,1}\cdots T_{i,n_i}]$ acting on a Hilbert space $\cH$ which admit   ${\bf S}$ as universal model, i.e.
 there is a Hilbert space $\cD$  such that  $\cH$ is jointly co-invariant for  all operators ${S}_{i,s}\otimes I_\cD$ and
 $$
 T_{i,s}^*=(S^*_{i,s}\otimes I_\cD)|_\cH,\quad i\in \{1,\ldots, k\} \ \text{ and } \  s\in \{1,\ldots, n_i\},
 $$
   consists of the pure elements of  a set ${\bf B}_\Lambda(\cH)$ which was called  the  regular $\Lambda$-polyball. 
  
  The goal of the present paper is to introduce  and study noncommutative Hardy spaces associated with the regular $\Lambda$-polyball,  to  develop a  functional calculus on noncommutative Hardy spaces   for the completely non-coisometric  (c.n.c.) $k$-tuples in ${\bf B}_\Lambda(\cH)$,  and to study the characteristic functions and the associated multi-analytic models  for the  c.n.c.  elements in the regular $\Lambda$-polyball.  In addition, we show that the characteristic function is a complete unitary invariant for the class of c.n.c. $k$-tuples in ${\bf B}_\Lambda(\cH)$.  These results extend the   corresponding classical results  of Sz.-Nagy--Foia\c s for contractions  and  the noncommutative versions  for row contractions.
  In the particular case when $n_1=\cdots=n_k=1$ and $\Lambda_{ij}=1$, we obtain a functional calculus and operator model theory in terms of characteristic functions for $k$-tuples of contractions satisfying Brehmer condition.
\end{abstract}

      \maketitle

\bigskip

\section*{Contents}
{\it

\quad Introduction

\begin{enumerate}
\item[1.] Preliminaries on  regular $\Lambda$-polyballs  and noncommutative Berezin transforms
   \item[2.]   Noncommutative Hardy spaces associated with regular $\Lambda$-polyballs
   \item[3.]  Functional calculus
   \item[4.] Free holomorphic functions on regular $\Lambda$-polyballs
    \item[5.]   Characteristic functions and multi-analytic  models
   \end{enumerate}

\quad References

}

\bigskip

\section*{Introduction}

In a recent paper \cite{Po-twisted}, inspired by  the work of De Jeu and Pinto \cite{DJP}, and J.~Sarkar \cite{S}, we studied the structure of  the $k$-tuples of doubly $\Lambda$-commuting row isometries and the $C^*$-algebras they generate from the point of view of   noncommutative multivariable operator theory.  

Given   row isometries $V_i:=[V_{i,1}\cdots V_{i,n_i}]$, $i\in \{1,\ldots, k\}$, we say that  $V:=(V_1,\ldots, V_k)$  is a  $k$-tuple of {\it doubly $\Lambda$-commuting row isometries}  if
$$ V_{i,s}^* V_{j,t}=\overline{\lambda_{ij}(s,t)}V_{j,t}V_{i,s}^*
$$
for every $i,j\in \{1,\ldots, k\}$ with $i\neq j$ and every $s\in \{1,\ldots, n_i\}$, $t\in \{1,\ldots, n_j\}$,
where
$\Lambda_{ij}:=[\lambda_{i,j}(s,t)]$ is an $n_i\times n_j$-matrix  with the entries in the torus $\TT:=\{z\in \CC: \  |z|=1\}$ and   $\Lambda_{j,i}=\Lambda_{i,j}^*$.

We obtained  Wold decompositions  and used them to  classify  the $k$-tuples of doubly $\Lambda$-commuting row isometries  up to a unitary equivalence. We proved that
there is a one-to-one correspondence  between the unitary equivalence classes of $k$-tuples of doubly $\Lambda$-commuting row isometries and the enumerations of  $2^k$ unitary equivalence classes of unital representations of the  twisted $\Lambda$-tensor algebras $\otimes_{i\in A^c}^{\Lambda}\cO_{n_i}$,  as $A$ is any subset of $\{1,\ldots, k\}$, where $\cO_{n_i}$ is the Cuntz algebra with $n_i$ generators (see \cite{Cu}).  The algebra  $\otimes_{i\in A^c}^\Lambda\cO_{n_i}$ can be seen  as a twisted tensor product of Cuntz  algebras. We remark that, when $n_1=\cdots =n_k=1$, the corresponding algebras  are   higher-dimensional
noncommutative tori  which are  studied in  noncommutative differential geometry (see  \cite{R}, \cite{Con}, \cite{Dav},  and the  appropriate references there in). We should mention that  $C^*$-algebras generated by isometries with twisted commutation relations have been
studied in the literature  in various particular cases (see \cite{JPS},  \cite{Pr},      \cite{K}, and \cite{W}).

We introduced  in \cite{Po-twisted} the  standard $k$-tuple  ${\bf S}:=({\bf S}_1,\ldots, {\bf S}_k)$ of  doubly
$\Lambda$-commuting   pure row isometries  ${\bf S}_i:=[S_{i,1}\cdots  S_{i,n_i}]$ acting on the Hilbert space $\ell^2(\FF_{n_1}^+\times \cdots \times \FF_{n_k}^+)$,  where $\FF_n^+$ is the unital free semigroup with $n$ generators,  and
 proved that  the universal  $C^*$-algebra generated by   a $k$-tuple of  doubly
$\Lambda$-commuting  row isometries is $*$-isomorphic to the $C^*$-algebra $C^*(\{S_{i,s}\})$.
The regular $\Lambda$-polyball ${\bf B}_\Lambda(\cH)$  was introduced as the set of all $k$-tuples of
   row contractions  $T_i=[T_{i,1}\ldots  T_{i,n_i}]$, i.e. $T_{i,1}T_{i,1}^*+\cdots +T_{i,n_i}T_{i,n_i}^*\leq I$,  such that
 \begin{equation*}
 T_{i,s} T_{j,t}=\lambda_{ij}(s,t)T_{j,t}T_{i,s}
\end{equation*}
for every $i,j\in \{1,\ldots, k\}$ with $i\neq j$ and every $s\in \{1,\ldots, n_i\}$, $t\in \{1,\ldots, n_j\}$,
and such that
$$
\Delta_{rT}(I):=
(id-\Phi_{rT_k})\circ\cdots \circ (id-\Phi_{rT_1})(I)\geq 0,\qquad r\in [0,1),
$$
where $\Phi_{rT_i}:B(\cH)\to B(\cH)$  is the completely positive linear map defined by $\Phi_{rT_i}(X):=\sum_{s=1}^{n_i} r^2T_{i,s}XT_{i,s}^*$.
  We proved that
a $k$-tuple  $T:=(T_1,\ldots, T_k)$ of row operators $T_i:=[T_{i,1}\ldots  T_{i,n_i}]$, acting on  a Hilbert space  $\cH$, admits ${\bf S}$ as universal model, i.e.
 there is a Hilbert space $\cD$  such that  $\cH$ is jointly co-invariant for  ${S}_{i,s}\otimes I_\cD$ and
 $$
 T_{i,s}^*=(S^*_{i,s}\otimes I_\cD)|_\cH,\quad i\in \{1,\ldots, k\} \ \text{ and } \  s\in \{1,\ldots, n_i\},
 $$
   if and only if  $T$ is a pure element of
${\bf B}_\Lambda(\cH)$. 

The goal of the present paper is to  continue the work in \cite{Po-twisted}  and  develop a multivariable functional calculus  for $k$-tuples of $\Lambda$-commuting row contractions   on noncommutative Hardy spaces associated with regular $\Lambda$-polyballs.  We  also study  the   characteristic functions   and  the associated multi-analytic models  for the elements of ${\bf B}_\Lambda(\cH)$.
Many of the techniques developed in  \cite{Po-twisted} and  \cite{Po-Berezin1} are refined and used in the present paper.

In Section 1, we present some preliminaries on noncommutative Berezin transforms associated with $\Lambda$-polyballs which are very useful in the  next sections. 
In Section 2, we introduce the noncommutative Hardy algebra $F^\infty({\bf B}_\Lambda)$ which can be seen as a noncommutative multivariable version of the Hardy algebra $H^\infty(\DD)$.
We prove that  $F^\infty ({\bf B}_\Lambda)$ is   WOT- (resp. SOT-, w*-) closed
  and
$$
F^\infty({\bf B}_\Lambda)=\overline{\cP(\{S_{i,s}\})}^{\text{\rm SOT}}= \overline{\cP(\{S_{i,s}\})}^{\text{\rm WOT}}=\overline{\cP(\{S_{i,s}\})}^{\text{\rm w*}},
$$
where $\cP(\{S_{i,s}\})$ is  the algebra of all polynomials in $S_{i,s}$ and the identity.
Moreover, we show that $F^\infty({\bf B}_\Lambda)$ is the sequential SOT-(resp. WOT-, w*-) closure of  $\cP(\{S_{i,s}\})$. Using noncommutative Berezin transforms associated with 
$\Lambda$-polyballs,  we prove that each element 
$A\in F^\infty({\bf B}_\Lambda)$ has a unique  formal Fourier representation
$$
\varphi(\{S_{i,s}\})=\sum_{(\beta_1,\ldots, \beta_k)\in  \FF_{n_1}^+\times\cdots \times \FF_{n_k}^+} c_{(\beta_1,\ldots, \beta_k)} S_{1,\beta_1}\ldots S_{k,\beta_k}
$$
such that,  for all $r\in [0,1)$,
  $\varphi(\{rS_{i,s}\})$ is in the $\Lambda$-polyball  algebra $\cA({\bf B}_\Lambda)$,
 the normed closed non-self-adjoint algebra generated by the isometries $S_{i,s}$ and the identity. Moreover, we prove that
   $$
A=\text{\rm SOT-}\lim_{r\to1}\varphi(\{rS_{i,s}\})
$$
and
$$
\|A\|=\sup_{0\leq r<1}\|\varphi(\{rS_{i,s}\})\|=\lim_{r\to 1}\|\varphi(\{rS_{i,s}\})\|.
$$

In Section 3, we prove the existence of an $F^\infty({\bf B}_\Lambda)$-functional calculus  for the completely non-coisometric (c.n.c)  elements  $T$ in the $\Lambda$-polyball ${\bf B}_\Lambda$ which extends the Sz.-Nagy--Foias functional calculus  for c.n.c. contractions \cite{SzFBK-book}  and the functional calculus for c.n.c row contractions \cite{Po-funct}. In this case, we  prove that if  $\varphi(\{S_{i,s}\})$ is the Fourier representation of $A\in F^\infty({\bf B}_\Lambda)$, then
$$
\Psi_T(A):=\text{\rm SOT-}\lim_{r\to 1} \varphi(\{rT_{i,s}\})
$$
exists and defines  a unital completely contractive homomorphism
 $\Psi_T: F^\infty({\bf B}_\Lambda)\to B(\cH)$ which  is WOT-(resp. SOT-, w*-) continuous on bounded sets.  
 
 Section 4 is dedicated to  the set  $Hol( {\bf B}_\Lambda^\circ )$ of free holomorphic functions on the open $\Lambda$-polyball ${\bf B}_\Lambda^\circ(\cH)$, which is the interior of ${\bf B}_\Lambda(\cH)$.
 We introduce the algebra  $H^\infty({\bf B}^\circ_\Lambda)$ of all $\varphi\in Hol( {\bf B}_\Lambda^\circ )$
such that
$$
\|\varphi\|_\infty:=\sup\|\varphi(\{X_{i,s}\})\|<\infty,
$$
where the supremum is taken over all $\{X_{i,s}\}\in {\bf B}_\Lambda^\circ(\cH)$ and any Hilbert space.    $H^\infty( {\bf B}_\Lambda^\circ )$ is a Banach algebra under pointwise multiplication and the norm $\|\cdot\|_\infty$ and has an  operator space structure in the sense of Ruan
 \cite{P-book}.  Using noncommutative Berezin transforms, we show that  the algebra of  bounded  free holomorphic functions 
  $H^\infty( {\bf B}_\Lambda^\circ )$ is  completely isometric isomorphic to the noncommutative Hardy algebra  $F^\infty({\bf B}_\Lambda)$ introduced in Section 2.
We also introduce  the algebra  $A( {\bf B}_\Lambda^\circ)$  of all functions $f\in  Hol( {\bf B}_\Lambda^\circ )$ such that the map
${\bf B}_\Lambda^\circ (\cH)\ni X\mapsto f(X)\in B(\cH)$
has a continuous extension to ${\bf B}_\Lambda(\cH)$ for any Hilbert space $\cH$.
It turns out  that $A( {\bf B}_\Lambda^\circ)$ is a Banach algebra with pointwise multiplication and the norm $\|\cdot\|_\infty$ and has an operator space structure.   We conclude this section by showing that
 $ A( {\bf B}_\Lambda^\circ)$ is completely isometric isomorphic to the noncommutative  
 $\Lambda$-polyball algebra  $\cA({\bf B}_\Lambda)$.  
 
 In Section 5, we show that  a $k$-tuple  ${T}=({ T}_1,\ldots, { T}_k)$  in  the noncommutative $\Lambda$-polyball ${\bf B}_\Lambda(\cH)$  admits a characteristic function if and only if
$$
{\Delta}_{{\bf S}\otimes I}(I -{ K_{T}}{ K_{T}^*})\geq 0,
$$
  where ${K_{T}}$ is the noncommutative Berezin kernel  associated with ${T}$ and
  $${\Delta}_{{\bf S}\otimes I}:=(id-\Phi_{{\bf S}_1\otimes I})\circ\cdots \circ (id-\Phi_{{\bf S}_k\otimes I}).
  $$
We provide a model theorem for the class of    completely non-coisometric $k$-tuple  of operators  in  
${\bf B}_\Lambda(\cH)$ which  admit characteristic functions, and     show  that the  characteristic function
   is a complete unitary invariant for this class of $k$-tuples.    These are generalizations of the corresponding classical results \cite{SzFBK-book} and of the noncommutative versions obtained in \cite{Po-charact}.
   
We remark that in the particular case when $n_1=\cdots=n_k=1$ and $\Lambda_{ij}=1$, we obtain a functional calculus and operator model theory   for $k$-tuples of contractions satisfying Brehmer condition \cite{Br} (see also \cite{SzFBK-book}).

\bigskip

\section{Preliminaries on  regular $\Lambda$-polyballs  and noncommutative Berezin transforms}

In this section,  we introduce the  standard $k$-tuple  ${\bf S}:=({\bf S}_1,\ldots, {\bf S}_k)$ of  doubly
$\Lambda$-commuting   pure row isometries  ${\bf S}_i:=[S_{i,1}\cdots  S_{i,n_i}]$ and  present some preliminaries results  on noncommutative Berezin transforms associated with $\Lambda$-polyballs.

For each $i,j\in \{1,\ldots, k\}$ with $i\neq j$, let $\Lambda_{ij}:=[\lambda_{i,j}(s,t)]$, where $s\in \{1,\ldots, n_i\}$ and $t\in \{1,\ldots, n_j\}$ be an $n_i\times n_j$-matrix  with the entries in the torus $\TT:=\{z\in \CC: \  |z|=1\}$, and assume that $\Lambda_{j,i}=\Lambda_{i,j}^*$. Given   row isometries $V_i:=[V_{i,1}\cdots V_{i,n_i}]$, $i\in \{1,\ldots, k\}$, we say that  $V=(V_1,\ldots, V_k)$  is a the $k$-tuple of {\it doubly $\Lambda$-commuting row isometries}  if
$$ V_{i,s}^* V_{j,t}=\overline{\lambda_{ij}(s,t)}V_{j,t}V_{i,s}^*
$$
for every $i,j\in \{1,\ldots, k\}$ with $i\neq j$ and every $s\in \{1,\ldots, n_i\}$, $t\in \{1,\ldots, n_j\}$.
We remark that the relation above implies that
$$ V_{i,s} V_{j,t}=\lambda_{ij}(s,t)V_{j,t}V_{i,s}.
$$

For each $i\in \{1,\ldots, k\}$, let $\FF_{n_i}^+$ be the unital free semigroup with generators $g_1^i,\ldots, g_{n_i}^i$ and neutral element $g_0^i$. The length of  $\alpha\in \FF_{n_i}^+$ is defined by $|\alpha|=0$ if $\alpha=g_0^i$ and $|\alpha|=m$ if $\alpha=g_{p_1}^i\cdots g_{p_m}^i\in \FF_{n_i}^+$, where $p_1,\ldots, p_m\in \{1,\ldots, n_i\}$.
If $T_i:=[T_{i,1}\cdots T_{i,n_i}]$,  we  use the notation
$T_{i,\alpha}:=T_{i,p_1}\cdots T_{i,p_m}$ and  $T_{i,g_0^i}:=I$.

Consider the Hilbert space $\ell^2(\FF_{n_1}^+\times\cdots \times \FF_{n_k}^+)$ with the standard basis  $\{\chi_{(\alpha_1,\ldots, \alpha_k)}\}$, where
 $\alpha\in \FF_{n_1}^+,\ldots \alpha_k\in \FF_{n_k}^+$.
For each  $i\in \{1,\ldots, k\}$ and $s\in \{1,\ldots, n_i\}$, we define the row operator  ${\bf S}_i:=[ S_{i,1}\cdots S_{i,n_i}]$, where $S_{i,s}$ is defined on $\ell^2(\FF_{n_1}^+\times\cdots \times \FF_{n_k}^+)$ by setting
\begin{equation}
\label{shift}
\begin{split}
&{S}_{i,s}\left( \chi_{(\alpha_1,\ldots, \alpha_k)}\right)\\
&\qquad \qquad :=\begin{cases} \chi_{(g_s^i\alpha_1,\alpha_2,\ldots, \alpha_k)}, &\quad \text{ if } i=1\\
\boldsymbol{\lambda}_{i,1}(s,\alpha_1)\cdots \boldsymbol{\lambda}_{i,i-1}(s,\alpha_{i-1})
\chi_{(\alpha_1,\ldots, \alpha_{i-1}, g_s^i \alpha_i, \alpha_{i+1},\ldots, \alpha_k)},&\quad \text{ if } i\in \{2,\ldots,k\}
\end{cases}
\end{split}
\end{equation}
for all $\alpha_1\in \FF_{n_1}^+,\ldots, \alpha_{k}\in \FF_{n_{k}}^+$,
where, for each   $j\in \{1,\ldots, k\}$,
$$
\boldsymbol{\lambda}_{i,j}(s, \beta)
:= \begin{cases}\prod_{b=1}^q\lambda_{i,j}(s,j_b),&\quad  \text{ if }
 \beta=g_{j_1}^j\cdots g_{j_q}^j\in \FF_{n_j}^+  \\
 1,& \quad  \text{ if } \beta=g_0^j.
 \end{cases}
$$
Let $i\in \{1,\ldots, k\}$ and $s\in \{1,\ldots, n_i\}$ and note that relation \eqref{shift} implies
\begin{equation}
\label{shift*}
\begin{split}
&{S}_{i,s}^*\left( \chi_{(\alpha_1,\ldots, \alpha_k)}\right)\\
&\qquad \qquad =\begin{cases}
\overline{\boldsymbol{\lambda}_{i,1}(s,\alpha_1)}\cdots \overline{\boldsymbol{\lambda}_{i,i-1}(s,\alpha_{i-1})}
\chi_{(\alpha_1,\ldots, \alpha_{i-1}, \beta_i, \alpha_{i+1},\ldots, \alpha_k)},&\quad \text{ if } \alpha_i=g_s^i \beta_i\\
0, &\quad \text{ otherwise }
\end{cases}
\end{split}
\end{equation}
for any  $\alpha_1\in \FF_{n_1}^+,\ldots, \alpha_{k}\in \FF_{n_{k}}^+$. Hence, we deduce that
\begin{equation*}
\begin{split}
&\sum_{s=1}^{n_i} {S}_{i,s}{S}_{i,s}^*\left( \chi_{(\alpha_1,\ldots, \alpha_k)}\right)\\
&\qquad \qquad =\begin{cases}
|\boldsymbol{\lambda}_{i,1}(s,\alpha_1)|^2\cdots |\boldsymbol{\lambda}_{i,i-1}(s,\alpha_{i-1})|^2
\chi_{(\alpha_1,\ldots, \alpha_{i-1}, \alpha_i, \alpha_{i+1},\ldots, \alpha_k)},&\quad \text{ if } |\alpha_i|\geq 1\\
0, &\quad \text{ otherwise }
\end{cases}\\
&\qquad \qquad =\begin{cases}
\chi_{(\alpha_1,\ldots, \alpha_k)},&\quad \text{ if } |\alpha_i|\geq 1\\
0, &\quad \text{ otherwise,}
\end{cases}
\end{split}
\end{equation*}
which shows that $[{S}_{i,1}\cdots {S}_{i, n_i}]$ is a  row isometry for every $i\in \{1,\ldots, k\}$.
In \cite{Po-twisted}, we showed that,  if $i,j\in \{1,\ldots, k\}$ with $i\neq j$ and any $s\in \{1,\ldots, n_i\}$, $t\in \{1,\ldots, n_j\}$, then
\begin{equation}\label{lac-S} {S}_{i,s}^* {S}_{j,t}=\overline{\lambda_{i,j}(s,t)}{S}_{j,t}{S}_{i,s}^*.
\end{equation}
 Consequently, ${\bf S}:=({\bf S}_1,\ldots, {\bf S}_k)$ is a $k$-tuple of doubly $\Lambda$-commuting row isometries.

  Given  row contractions $T_i:=[T_{i,1}\cdots T_{i,n_i}]$, $i\in \{1,\ldots, k\}$,  acting on a Hilbert space $\cH$, we say that  $T=(T_1,\ldots, T_k)$  is a $k$-tuple of {\it  $\Lambda$-commuting row contractions}  if
\begin{equation}
\label{commuting}
 T_{i,s} T_{j,t}=\lambda_{ij}(s,t)T_{j,t}T_{i,s}
\end{equation}
for every $i,j\in \{1,\ldots, k\}$ with $i\neq j$ and every
 $s\in \{1,\ldots, n_i\}$, $t\in \{1,\ldots, n_j\}$.
    We say that $T$ is in the {\it regular $\Lambda$-polyball}, which we denote by ${\bf B}_\Lambda(\cH)$, if T is  a $\Lambda$-commuting  tuple and
$$
\Delta_{rT}(I):=
(id-\Phi_{rT_k})\circ\cdots \circ (id-\Phi_{rT_1})(I)\geq 0,\qquad r\in [0,1),
$$
where $\Phi_{rT_i}:B(\cH)\to B(\cH)$  is the completely positive linear map defined by $\Phi_{rT_i}(X):=\sum_{s=1}^{n_i} r^2T_{i,s}XT_{i,s}^*$.
We remark that, due to the $\Lambda$-commutation relation \eqref{commuting}, we have  $\Phi_{T_i}\circ \Phi_{T_j}(X)=\Phi_{T_j}\circ \Phi_{T_i}(X)$ for any $i,j\in \{1,\ldots,k\}$ and $X\in B(\cH)$.

Let $T=(T_1,\ldots, T_k)$ be  a $k$-tuple in the {\it regular $\Lambda$-polyball} ${\bf B}_\Lambda(\cH)$. We define the noncommutative Berezin kernel
$$
K_{T}:\cH\to \ell^2(\FF_{n_1}^+\times\cdots \times \FF_{n_k}^+)\otimes \cD(T),
$$
 by setting
$$
K_{T}h :=\sum_{\beta_1\in \FF_{n_1}^+,\ldots, \beta_k\in \FF_{n_k}^+}
\chi_{(\beta_1,\ldots, \beta_k)}\otimes   \Delta_{T}(I)^{1/2}T_{k,\beta_k}^*\cdots T_{1,\beta_1}^*h,\qquad h\in \cD(T),
$$
where $ \Delta_{T}(I):=(id-\Phi_{T_k})\circ\cdots \circ (id-\Phi_{T_1})(I)$ and $\cD(T):=\overline{ \Delta_{T}(I)\cH}$.

The  first theorem is an extension of the corresponding result from  \cite{Po-twisted} for pure $k$-tuples in ${\bf B}_\Lambda(\cH)$.
\begin{theorem}  \label{Berezin} Let $T=(T_1,\ldots, T_k)$ be  a  $k$-tuple in the regular $\Lambda$-polyball ${\bf B}_\Lambda(\cH)$.  Then the following statements hold.
\begin{enumerate}
\item[(i)] The noncommutative Berezin kernel $K_T$ is a contraction and
$$
K_T^*K_T=\lim_{p_k\to\infty}\ldots \lim_{p_1\to\infty}(id-\Phi_{T_k}^{p_k})\circ\cdots \circ (id-\Phi_{T_1}^{p_1})(I),
$$
where the limits are in the weak operator theory.
\item[(ii)] For every $i\in \{1,\ldots, k\}$ and $s\in \{1,\ldots, n_i\}$,
$$
K_T T_{i,s}^*=\left(S_{i,s}^*\otimes I_{\cD(T)}\right) K_T.
$$

\end{enumerate}

\end{theorem}
\begin{proof}
For each $i\in \{1,\ldots, k\}$, we set
$$\Delta_{(T_i, T_{i-1},\ldots, T_1)}(I):=(id-\Phi_{T_{i}})\circ\cdots \circ (id-\Phi_{T_1})(I)
$$
and remark that, due to the fact that $T_i$ is a row contraction,  $A_i:=\lim_{q_i\to \infty}\Phi_{T_i}^{q_i+1}(I)$ exists   in the weak operator theory.
Using the fact that  $\Phi_{T_i}\circ \Phi_{T_j}(X)=\Phi_{T_j}\circ \Phi_{T_i}(X)$ for all $i,j\in \{1,\ldots,k\}$ and $X\in B(\cH)$, we deduce that
\begin{equation*}
\begin{split}
\sum_{q_k=0}^\infty \Phi_{T_k}^{q_k} [\Delta_{(T_k,\ldots, T_1)}(I)]
& =
\lim_{p_k\to \infty} \sum_{q_k=0}^{p_k}\left\{ \Phi_{T_k}^{q_k}
[\Delta_{(T_{k-1},\ldots, T_1)}(I)]
-\Phi_{T_k}^{q_k+1}[\Delta_{(T_{k-1},\ldots, T_1)}(I)]\right\}\\
&= \Delta_{(T_{k-1},\ldots, T_1)}(I)
-\lim_{p_k\to \infty} \Phi_{T_k}^{p_k+1}[\Delta_{(T_{k-1},\ldots, T_1)}(I))]\\
&= \Delta_{(T_{k-1},\ldots, T_1)}(I)
-\Delta_{(T_{k-1},\ldots, T_1)}\left(\lim_{p_k\to \infty} \Phi_{T_k}^{p_k+1}(I)\right)\\
&=
\Delta_{(T_{k-1},\ldots, T_1)}(I-A_k).
\end{split}
\end{equation*}

Consequently, we deduce that
\begin{equation*}
\begin{split}
&\sum_{q_{k-1}=0}^\infty \Phi_{T_{k-1}}^{q_{k-1}}\left(\sum_{q_{k}=0}^\infty \Phi_{T_k}^{q_k} [\Delta_{(T_k,\ldots, T_1)}(I)] \right)\\
&\qquad=\sum_{q_{k-1}=0}^\infty \Phi_{T_{k-1}}^{q_{k-1}}\left(\Delta_{(T_{k-1},\ldots, T_1)}(I-A_k)\right)\\
&\qquad =
\lim_{p_{k-1}\to \infty} \sum_{q_{k-1}=0}^{p_{k-1}}\left\{ \Phi_{T_{k-1}}^{q_{k-1}}
[\Delta_{(T_{k-2},\ldots, T_1)}(I-A_k)]
-\Phi_{T_{k-1}}^{q_{k-1}+1}[\Delta_{(T_{k-2},\ldots, T_1)}(I-A_k)]\right\}\\
&\qquad= \Delta_{(T_{k-2},\ldots, T_1)}(I-A_k)
-\lim_{p_{k-1}\to \infty} \Phi_{T_{k-1}}^{p_{k-1}+1}[\Delta_{(T_{k-2},\ldots, T_1)}(I-A_k)]\\
&\qquad = \Delta_{(T_{k-2},\ldots, T_1)}(I-A_k)
-\Delta_{(T_{k-2},\ldots, T_1)}\left((I-A_k)\lim_{p_{k-1}\to \infty} \Phi_{T_{k-1}}^{p_{k-1}+1}(I)\right)\\
&\qquad =
\Delta_{(T_{k-2},\ldots, T_1)}[(I-A_k)(I-A_{k-1})].
\end{split}
\end{equation*}
Continuing this process, we obtain
$$
\sum_{q_1=0}^\infty \Phi_{T_1}^{q_1}\left(\sum_{q_2=0}^\infty \Phi_{T_2}^{q_2}\left(\cdots \sum_{q_k=0}^\infty \Phi_{T_k}^{q_k} [\Delta_{(T_k,\ldots, T_1)}(I)]\cdots \right)\right)=(I-A_k)\cdots (I-A_1),
$$
where the convergence of the series is in the weak operator topology.
Since   we can rearrange the series  of positive terms,  we obtain
$$
\sum_{q_1,\ldots, q_k=0}^\infty \Phi_{T_1}^{q_1}\circ \cdots \circ \Phi_{T_k}^{q_k}[\Delta_{(T_k,\ldots, T_1)}(I)]=(I-A_k)\cdots (I-A_1).
$$
Using this relation, one can see that

\begin{equation*}
\begin{split}
\left<K_T^*K_Th,h\right>&=  \left<\sum_{\beta_1\in \FF_{n_1}^+,\ldots, \beta_k\in \FF_{n_k}^+}
   T_{1,\beta_1} \cdots T_{k,\beta_k} \Delta_{T}(I)T_{k,\beta_k}^*\cdots T_{1,\beta_1}^* h, h\right>\\
   &=\left<(I-A_k)\cdots (I-A_1)h,h\right>
\end{split}
\end{equation*}
for any $h\in \cH$, which proves item (i).

Now, we prove item (ii).
Note that, for every $h,h'\in \cH$,
\begin{equation*}
\begin{split}
\left<K_T T_{i,s}^*h,  \chi_{(\alpha_1,\ldots, \alpha_k)}\otimes h'\right>
&=\left< \sum_{\beta_1\in \FF_{n_1}^+,\ldots, \beta_k\in \FF_{n_k}^+}
\chi_{(\beta_1,\ldots, \beta_k)}\otimes   \Delta_{T}(I)^{1/2}T_{k,\beta_k}^*\cdots T_{1,\beta_1}^* T_{i,s}^*h,
 \chi_{(\alpha_1,\ldots, \alpha_k)}\otimes h'
\right>\\
&=\left<\Delta_{T}(I)^{1/2}T_{k,\alpha_k}^*\cdots T_{1,\alpha_1}^*T_{i,s}^* h,h'
\right>\\
&=
\left< h, T_{i,s} T_{1,\alpha_1}\cdots T_{i-1,\alpha_{i-1}}T_{i,\alpha_i}\cdots T_{k,\alpha_k}\Delta_{T}(I)^{1/2}h'\right>\\
&= \overline{\boldsymbol{\lambda}_{i,1}(s,\alpha_1)}\cdots \overline{\boldsymbol{\lambda}_{i,i-1}(s,\alpha_{i-1})}
\left<h,   T_{1,\alpha_1}\cdots T_{i-1,\alpha_{i-1}}T_{i,g_s^i\alpha_i} \cdots T_{k,\alpha_k}\Delta_{T}(I)^{1/2}h'\right>
\end{split}
\end{equation*}
for all $\alpha_1\in \FF_{n_1}^+,\ldots, \alpha_k\in \FF_{n_k}^+$ where, for all   $j\in \{1,\ldots, k\}$,
\begin{equation}\label{laij}
\boldsymbol{\lambda}_{i,j}(s, \beta)
:= \begin{cases}\prod_{b=1}^q\lambda_{i,j}(s,j_b)&\quad  \text{ if }
 \beta=g_{j_1}^j\cdots g_{j_q}^j\in \FF_{n_j}^+  \\
 1& \quad  \text{ if } \beta=g_0^j.
 \end{cases}
\end{equation}
   Due to the definition of the noncommutative Berezin kernel $K_T$ and using relation \eqref{shift*},  we obtain
\begin{equation*}
\begin{split}
&\left< \left(S_{i,s}^*\otimes I\right) K_T h, \chi_{(\alpha_1,\ldots, \alpha_k)}\otimes h'\right>\\
& =
\left< S_{i,s}^*(\chi_{(\alpha_1,\ldots, \alpha_{i-1}, g_s^i \alpha_i, \alpha_{i+1},\ldots,  \alpha_k)})
\otimes \Delta_{T}(I)^{1/2} T_{k,\alpha_k}^*\cdots T_{i+1,\alpha_{i+1}} ^*T_{i, g_s^i\alpha_i}^*T_{i-1,\alpha_{i-1}}^*\cdots T_{1,\alpha_1}^*h, \chi_{(\alpha_1,\ldots, \alpha_k)}\otimes h'\right>\\
&=
\overline{\boldsymbol{\lambda}_{i,1}(s,\alpha_1)}\cdots \overline{\boldsymbol{\lambda}_{i,i-1}(s,\alpha_{i-1})}
\left<h,   T_{1,\alpha_1}\cdots T_{i-1,\alpha_{i-1}}T_{i,g_s^i\alpha_i} \cdots T_{k,\alpha_k}\Delta_{T}(I)^{1/2}h'\right>.
\end{split}
\end{equation*}
Consequently, we obtain
$$
\left< \left(S_{i,s}^*\otimes I\right) K_T h, \chi_{(\alpha_1,\ldots, \alpha_k)}\otimes h'\right>
=
\left<h,   T_{1,\alpha_1}\cdots T_{i-1,\alpha_{i-1}}T_{i,g_s^i\alpha_i} \cdots T_{k,\alpha_k}\Delta_{T}(I)^{1/2}h'\right>
$$
and conclude that  item (ii) holds. The proof is complete.
\end{proof}

 Note that due to the doubly  $\Lambda$-commutativity  relations  \eqref{lac-S} satisfied by the standard shift ${\bf S}=({\bf S}_1,\ldots, {\bf S}_n)$  and the fact that $S_{i,s}^* S_{i,t}=\delta_{st}I$ for every $i\in \{1,\ldots, k\}$ and $s,t\in \{1,\ldots, n_i\}$,  and every polynomial in
$\{S_{i,s}\}$ and $ \{S_{i,s}^*\}$ is a finite sum the form
$$
p(\{S_{i,s}\}, \{S_{i,s}^*\})= \sum a_{(\alpha_1,\ldots, \alpha_p,\beta_1,\ldots, \beta_m)}S_{i_1,\alpha_1}\cdots S_{i_p,\alpha_p}S_{j_1,\beta_1}^*\cdots S_{j_m,\beta_m}^*,
$$
where $\alpha_1\in \FF_{n_{i_1}}^+,\ldots, \alpha_p\in \FF_{n_{i_p}}^+$ and $\beta_1\in \FF_{n_{j_1}}^+,\ldots, \beta_m\in \FF_{n_{j_m}}^+$.
We define
$$
p(\{T_{i,s}\}, \{T_{i,s}^*\}):=\sum a_{(\alpha_1,\ldots, \alpha_p,\beta_1,\ldots, \beta_m)}T_{i_1,\alpha_1}\cdots T_{i_p,\alpha_p}T_{j_1,\beta_1}^*\cdots T_{j_m,\beta_m}^*
$$
and note that the definition is correct due to the following von Neumann inequality obtained in \cite{Po-twisted}, i.e.
$$
\|p(\{T_{i,s}\}, \{T_{i,s}^*\})\|\leq  \|p(\{S_{i,s}\}, \{S_{i,s}^*\})\|
$$
for every  $k$-tuple  $T=(T_1,\ldots, T_k)$  in the {\it regular $\Lambda$-polyball}, which extends the classical result \cite{vN} and the noncommutative version for row contractions \cite{Po-von}.

The $\Lambda$-polyball algebra $\cA({\bf B}_\Lambda)$ is the normed closed non-self-adjoint   algebra  generated by  the isometries $S_{i,s}$, where $i\in\{1,\ldots, k\}$ and $j\in \{1,\ldots, n_i\}$,  and the identity.  We denote by  $C^*(\{S_{i,s}\})$ the $C^*$-algebra generated by the isometries $S_{i,s}$
We prove in \cite{Po-twisted} that
if $T\in {\bf B}_\Lambda (\cH)$, then the map
$$
\Psi_T(f):=\lim_{r\to 1}K_{rT}^*[f\otimes I] K_{rT}, \qquad f\in C^*(\{S_{i,s}\}),
$$
where the limit is in the operator norm topology, is a is completely contractive linear map. Moreover, its restriction to the  $\Lambda$-polyball algebra $\cA({\bf B}_\Lambda)$ is a completely contractive homomorphism.
If, in addition, $T$ is a pure $k$-tuple, i.e., for each $i\in \{1,\ldots, k\}$,  $\Phi_{T_i}^p(I)\to 0$, as $p\to\infty$, then $\Psi_T(f)=K_{T}^*[f\otimes I] K_{T}$. We call the map $\Psi_T$ the   {\it noncommutative Berezin transform} at $T$ associated with the $\Lambda$-polyball.

 \bigskip

\section{Noncommutative Hardy spaces associated with regular $\Lambda$-polyballs}

In this section, we introduce the noncommutative Hardy algebra $F^\infty({\bf B}_\Lambda)$, which can be seen as a noncommutative multivariable version of the Hardy algebra $H^\infty(\DD)$, and prove some basic properties.

According to relations \eqref{shift} and \eqref{laij}, for each $i\in \{1,\ldots, k\}$ and $\boldsymbol \alpha:=(\alpha_1,\ldots, \alpha_k)\in \FF_{n_1}^+\times\cdots \times \FF_{n_k}^+$, we have
$$
S_{i,g_s^i}(\chi_{\boldsymbol \alpha})=\boldsymbol\mu_i(g_s^i,\boldsymbol\alpha) \chi_{\alpha_1,\ldots, \alpha_{i-1}, g_s^i \alpha_i, \alpha_{i+1},\ldots, \alpha_k)},
$$
where
$$
\boldsymbol\mu_i(g_s^i,\boldsymbol\alpha):=
\boldsymbol{\lambda}_{i,1}(s,\alpha_1)\cdots \boldsymbol{\lambda}_{i,i-1}(s,\alpha_{i-1}).
$$
Consequently, if $\gamma_i:=g_{i_1}^i\cdots g_{i_p}^i\in \FF_{n_i}^+$, then
$$
S_{i,\gamma^i}(\chi_{\boldsymbol \alpha})=\boldsymbol\mu_i(\gamma_i,\boldsymbol\alpha) \chi_{\alpha_1,\ldots, \alpha_{i-1}, \gamma_i \alpha_i, \alpha_{i+1},\ldots, \alpha_k)},
$$
where
$$
\boldsymbol\mu_i(\gamma_i,\boldsymbol\alpha):=
\boldsymbol{\mu}_{i}(g_{i_1}^i, \boldsymbol\alpha)\cdots \boldsymbol{\mu}_{i}(g_{i_p}^i,\boldsymbol\alpha).
$$
Given $\boldsymbol \gamma:=(\gamma_1,\ldots, \gamma_k)\in  \FF_{n_1}^+\times\cdots \times \FF_{n_k}^+$, we deduce that
$$
S_{1,\gamma_1}\cdots S_{k,\gamma_k}(\chi_{\boldsymbol \alpha})=\boldsymbol\mu(\boldsymbol\gamma,\boldsymbol\alpha) \chi_{(\gamma_1\alpha_1,\ldots, \gamma_k\alpha_k)}
$$
where
$$
\boldsymbol\mu(\boldsymbol\gamma,\boldsymbol\alpha):=
\boldsymbol\mu_1(\gamma_1,\boldsymbol\alpha)\cdots \boldsymbol\mu_k(\gamma_k,\boldsymbol\alpha).
$$

Let $\{c_{(\beta_1,\ldots, \beta_k)}\}_{(\beta_1,\ldots, \beta_k)\in  \FF_{n_1}^+\times\cdots \times \FF_{n_k}^+}$ be a sequence of complex numbers such $\sum |c_{(\beta_1,\ldots, \beta_k)}|^2<\infty$ and consider the formal series
$$\varphi(\{S_{i,s}\}):=\sum_{(\beta_1,\ldots, \beta_k)\in  \FF_{n_1}^+\times\cdots \times \FF_{n_k}^+} c_{(\beta_1,\ldots, \beta_k)} S_{1,\beta_1}\ldots S_{1,\beta_k}.
$$
Set ${\bf g}_0:=(g_0^1,\ldots, g_0^k)$ and note that $ \boldsymbol \mu(\boldsymbol \beta, \bf{g}_0)\in \TT$ and
\begin{equation*}
\begin{split}
\varphi(\{S_{i,s}\})(\chi_{{\bf g}_0})&:=
\sum_{(\beta_1,\ldots, \beta_k)\in  \FF_{n_1}^+\times\cdots \times \FF_{n_k}^+} c_{(\beta_1,\ldots, \beta_k)} S_{1,\beta_1}\ldots S_{1,\beta_k}(\chi_{{\bf g}_0})\\
&=
\sum_{(\beta_1,\ldots, \beta_k)\in  \FF_{n_1}^+\times\cdots \times \FF_{n_k}^+} c_{(\beta_1,\ldots, \beta_k)} \boldsymbol \mu(\boldsymbol \beta, \bf{g}_0)\chi_{(\beta_1,\ldots, \beta_k)}
\end{split}
\end{equation*}
 is an element in $\ell^2(\FF_{n_1}^+\times\cdots \times \FF_{n_k}^+)$.
 Similarly,  for each $\boldsymbol \gamma:=(\gamma_1,\ldots, \gamma_k)\in  \FF_{n_1}^+\times\cdots \times \FF_{n_k}^+$, we have $\boldsymbol \mu(\boldsymbol \beta, \boldsymbol\gamma)\in \TT$ and
 $$
 \varphi(\{S_{i,s}\})(\chi_{\boldsymbol \gamma})=\sum_{(\beta_1,\ldots, \beta_k)\in  \FF_{n_1}^+\times\cdots \times \FF_{n_k}^+} c_{(\beta_1,\ldots, \beta_k)} \boldsymbol \mu(\boldsymbol \beta, \boldsymbol\gamma)\chi_{(\beta_1\gamma_1,\ldots, \beta_k\gamma_k)}
 $$
is an element in $\ell^2(\FF_{n_1}^+\times\cdots \times \FF_{n_k}^+)$.
Now,  let $\cP$ be the linear span of the vectors $\{\chi_{\boldsymbol \gamma}\}_{ \boldsymbol\gamma}$, assume that
$$\sup_{p\in \cP, \|p\|\leq 1} \|\varphi(\{S_{i,s}\})p\|<\infty.
$$
In this case, there is a unique operator $A\in B(\ell^2(\FF_{n_1}^+\times\cdots \times \FF_{n_k}^+))$ such that $Ap=\varphi(\{S_{i,s}\})p$ for any $p\in \cP$. We say that $\varphi(\{S_{i,s}\})$ is the formal Fourier series associated  $A$.
We denote by $F^\infty({\bf B}_\Lambda)$ the set of all operators $A$ obtained in this manner.

\begin{theorem}   \label{densities} Let   $\cP(\{S_{i,s}\})$ be  the algebra of all polynomials in $S_{i,s}$ and the identity, where $i\in \{1,\ldots, k\}$,  and $s\in \{1,\ldots, n_i\}$. Then  the noncommutative Hardy  space $F^\infty ({\bf B}_\Lambda)$ is   WOT- (resp. SOT-, w*-) closed
  and
$$
F^\infty({\bf B}_\Lambda)=\overline{\cP(\{S_{i,s}\})}^{\text{\rm SOT}}= \overline{\cP(\{S_{i,s}\})}^{\text{\rm WOT}}=\overline{\cP(\{S_{i,s}\})}^{\text{\rm w*}}.
$$
Moreover, $F^\infty({\bf B}_\Lambda)$ is the sequential SOT-(resp. WOT-, w*-) closure of  $\cP(\{S_{i,s}\})$.
\end{theorem}
\begin{proof}
First, we prove that the noncommutative Hardy  space $F^\infty({\bf B}_\Lambda)$ is WOT- (resp. SOT-) closed.
Let $\{A_\iota\}_\iota$ be a net in $F^\infty({\bf B}_\Lambda)$ and assume that WOT-$\lim_\iota A_\iota =A\in B(\ell^2(\FF_{n_1}^+\times\cdots \times \FF_{n_k}^+))$
If
$\sum_{(\beta_1,\ldots, \beta_k)\in  \FF_{n_1}^+\times\cdots \times \FF_{n_k}^+} c_{(\beta_1,\ldots, \beta_k)}^\iota S_{1,\beta_1}\ldots S_{1,\beta_k}
$
is the formal Fourier series of $A_\iota$, then
\begin{equation*}
\left<A \chi_{{\bf g}_0}, \chi_{(\beta_1,\ldots, \beta_k)}\right>=
\lim_\iota\left<A_\iota \chi_{{\bf g}_0}, \chi_{(\beta_1,\ldots, \beta_k)}\right>=\lim_\iota c_{(\beta_1,\ldots, \beta_k)}^\iota \boldsymbol\mu(\boldsymbol \beta,{\bf g}_0).
\end{equation*}
Define $c_{(\beta_1,\ldots, \beta_k)}:=\frac{1}{\boldsymbol\mu(\boldsymbol \beta,{\bf g}_0)}
\left<A \chi_{{\bf g}_0}, \chi_{(\beta_1,\ldots, \beta_k)}\right>$ and note that
$\lim_\iota c_{(\beta_1,\ldots, \beta_k)}^\iota=c_{(\beta_1,\ldots, \beta_k)}$.
On the other hand, we have
\begin{equation*}\begin{split}
\left<A \chi_{(\gamma_1,\ldots, \gamma_k)}, \chi_{(\beta_1\gamma_1,\ldots, \beta_k\gamma_k)}\right>&=
\lim_\iota\left<A_\iota \chi_{(\gamma_1,\ldots, \gamma_k)}, \chi_{(\beta_1\gamma_1,\ldots, \beta_k\gamma_k)}\right>\\
&=\lim_\iota c_{(\beta_1,\ldots, \beta_k)}^\iota \boldsymbol\mu(\boldsymbol \beta,\boldsymbol\gamma)\\
&=c_{(\beta_1,\ldots, \beta_k)}\mu(\boldsymbol \beta,\boldsymbol\gamma).
\end{split}
\end{equation*}
Note that
$$
\sum_{(\beta_1,\ldots, \beta_k)\in  \FF_{n_1}^+\times\cdots \times \FF_{n_k}^+} |c_{(\beta_1,\ldots, \beta_k)}|^2
=\sum_{(\beta_1,\ldots, \beta_k)\in  \FF_{n_1}^+\times\cdots \times \FF_{n_k}^+}
|\left<A \chi_{{\bf g}_0}, \chi_{(\beta_1,\ldots, \beta_k)}\right>|^2=\|A \chi_{{\bf g}_0}\|^2<\infty
$$
and consider the formal series
$$
\varphi(\{S_{i,s}\}):=\sum_{(\beta_1,\ldots, \beta_k)\in  \FF_{n_1}^+\times\cdots \times \FF_{n_k}^+} c_{(\beta_1,\ldots, \beta_k)} S_{1,\beta_1}\ldots S_{1,\beta_k}.
$$
Using the results above, one can see that
\begin{equation*}
\begin{split}
\left<A \chi_{(\gamma_1,\ldots, \gamma_k)}, \chi_{(\alpha_1,\ldots, \alpha_k)}\right>
&=\lim_\iota \left<A_\iota\chi_{(\gamma_1,\ldots, \gamma_k)}, \chi_{(\alpha_1,\ldots, \alpha_k)}\right>\\
&=\lim_\iota \left<  \sum_{(\beta_1,\ldots, \beta_k)\in  \FF_{n_1}^+\times\cdots \times \FF_{n_k}^+} c_{(\beta_1,\ldots, \beta_k)}^\iota S_{1,\beta_1}\ldots S_{1,\beta_k} \chi_{(\gamma_1,\ldots, \gamma_k)}, \chi_{(\alpha_1,\ldots, \alpha_k)}\right>\\
&=\lim_\iota \left<  \sum_{(\beta_1,\ldots, \beta_k)\in  \FF_{n_1}^+\times\cdots \times \FF_{n_k}^+} c_{(\beta_1,\ldots, \beta_k)}^\iota \boldsymbol \mu(\boldsymbol \beta, \boldsymbol\gamma)\chi_{(\beta_1\gamma_1,\ldots, \beta_k\gamma_k)}, \chi_{(\alpha_1,\ldots, \alpha_k)}\right>\\
&=\begin{cases}\lim_\iota  c_{(\beta_1,\ldots, \beta_k)}^\iota \boldsymbol \mu(\boldsymbol \beta, \boldsymbol\gamma),& \ \text{ if} \  (\alpha_1,\ldots,\alpha_k)=(\beta_1\gamma_1,\ldots, \beta_k\gamma_k)\\
0,& \  \text{ otherwise}
\end{cases}
\\
&=\begin{cases} c_{(\beta_1,\ldots, \beta_k)} \boldsymbol \mu(\boldsymbol \beta, \boldsymbol\gamma),& \ \text{ if} \  (\alpha_1,\ldots,\alpha_k)=(\beta_1\gamma_1,\ldots, \beta_k\gamma_k)\\
0,& \  \text{ otherwise}
\end{cases}
\\
&=
\left<\varphi(\{S_{i,s}\}) \chi_{(\gamma_1,\ldots, \gamma_k)}, \chi_{(\alpha_1,\ldots, \alpha_k)}\right>
\end{split}
\end{equation*}
for all $(\gamma_1,\ldots, \gamma_k), (\alpha_1,\ldots, \alpha_k)\in  \FF_{n_1}^+\times\cdots \times \FF_{n_k}^+$.
Consequently,  we have
$$\left<A p, \chi_{(\alpha_1,\ldots, \alpha_k)}\right>=
\left<\varphi(\{S_{i,s}\}) p, \chi_{(\alpha_1,\ldots, \alpha_k)}\right>
$$
for all $p\in \cP$. Hence, we deduce that
$$
\|Ap\|^2=\sum_{(\alpha_1,\ldots, \alpha_k)\in  \FF_{n_1}^+\times\cdots \times \FF_{n_k}^+}
|\left<A p, \chi_{(\alpha_1,\ldots, \alpha_k)}\right>|^2=\|\varphi(\{S_{i,s}\}) p\|^2
$$
which implies
$\sup_{p\in \cP, \|p\|\leq 1} \|\varphi(\{S_{i,s}\})p\|=\|A\|.
$
This shows that $A\in F^\infty({\bf B}_\Lambda)$ and $\varphi(\{S_{i,s}\})$ is its formal Fourier representation.

Now, we prove that  any operator in $F^\infty({\bf B}_\Lambda)$  is the   SOT-limit  of a sequence of  polynomials in $S_{i,s}$  and the identity.
For each $m\in \ZZ$, define the completely contractive projection
$Q_m: B(\ell^2(\FF_{n_1}^+\times\cdots \times \FF_{n_k}^+))\to B(\ell^2(\FF_{n_1}^+\times\cdots \times \FF_{n_k}^+))$ by setting
$$
Q_m(T):=\sum_{n\geq \max\{0,-m\}}P_nTP_{n+m},
$$
where $P_n$, $n\geq 0$, is the orthogonal projection of $\ell^2(\FF_{n_1}^+\times\cdots \times \FF_{n_k}^+)$ onto the span of all vectors $\chi_{(\beta_1,\ldots, \beta_k)}$ such that $|\beta_1|+\cdots +|\beta_k|=n$, where $\beta_i\in \FF_{n_i}^+$.
Consider the the Cesaro operators on $B(\ell^2(\FF_{n_1}^+\times\cdots \times \FF_{n_k}^+))$ defined by
$$
C_n(T):=\sum_{|m|<n}\left(1-\frac{|m|}{n}\right) Q_m(T),\qquad n\geq 1.
$$
One can easily see that these operators are completely contractive and SOT-$\lim_{n\to\infty} C_n(T)=T$.
Now, let $T\in F^\infty({\bf B}_\Lambda)$ have the formal Fourier representation
$$
\sum_{(\beta_1,\ldots, \beta_k)\in  \FF_{n_1}^+\times\cdots \times \FF_{n_k}^+} c_{(\beta_1,\ldots, \beta_k)} S_{1,\beta_1}\ldots S_{1,\beta_k}.
$$
Using the definition of the isometries $S_{i,s}$ we deduce that
$$
P_{n+m}TP_m=\left(\sum_{{(\beta_1,\ldots, \beta_k)\in  \FF_{n_1}^+\times\cdots \times \FF_{n_k}^+}\atop{|\beta_1|+\cdots +|\beta_k|=n}} c_{(\beta_1,\ldots, \beta_k)} S_{1,\beta_1}\ldots S_{1,\beta_k}\right)P_m
$$
for all $n,m\geq 0$.  On the other hand, we have $P_mTP_{n+m}=0$ if $n\geq 1$ and $m\geq 0$.
Consequently, we have
$$
C_n(T)=\sum_{0\leq p\leq n-1}\left( 1-\frac{p}{n}\right)
\left(\sum_{{(\beta_1,\ldots, \beta_k)\in  \FF_{n_1}^+\times\cdots \times \FF_{n_k}^+}\atop{|\beta_1|+\cdots +|\beta_k|=p}} c_{(\beta_1,\ldots, \beta_k)} S_{1,\beta_1}\ldots S_{1,\beta_k}\right)
$$
and SOT-$\lim_{n\to\infty} C_n(T)=T$. This shows that $T$ is the   SOT-limit  of a sequence of  polynomials in $S_{i,s}$  and the identity. Consequently, $T$ is also the WOT-(resp. w*-) limit  of a sequence of  polynomials in $S_{i,s}$  and the identity.  Denoting by $\cP(\{S_{i,s}\})$ the algebra of all polynomials in $S_{i,s}$  and the identity, we deduce that
$$
F^\infty({\bf B}_\Lambda)\subset \overline{\cP(\{S_{i,s}\})}^{\text{\rm SOT}}\subset \overline{\cP(\{S_{i,s}\})}^{\text{\rm WOT}}.
$$
Since $\cP(\{S_{i,s}\})\subset F^\infty({\bf B}_\Lambda)$  and $F^\infty({\bf B}_\Lambda)$ is WOT-closed, we have
$\overline{\cP(\{S_{i,s}\})}^{\text{\rm WOT}}\subset F^\infty({\bf B}_\Lambda)$.
Therefore,
$$
F^\infty({\bf B}_\Lambda)=\overline{\cP(\{S_{i,s}\})}^{\text{\rm SOT}}= \overline{\cP(\{S_{i,s}\})}^{\text{\rm WOT}}.
$$
Due to the results above, we also have
$F^\infty ({\bf B}_\Lambda)\subset \overline{\cP(\{S_{i,s}\})}^{\text{\rm w*}}$.
Moreover, since $F^\infty ({\bf B}_\Lambda)$ is a convex subset of $B(\ell^2(\FF_{n_1}^+\times\cdots \times \FF_{n_k}^+))$, we know that $F^\infty ({\bf B}_\Lambda)$ is w*-closed if and only if it is WOT sequential closed. Due to the results above, we conclude that $F^\infty ({\bf B}_\Lambda)$ is
w*-closed. Since $\cP(\{S_{i,s}\})\subset F^\infty({\bf B}_\Lambda)$, we have
$\overline{\cP(\{S_{i,s}\})}^{\text{\rm w*}}\subset F^\infty({\bf B}_\Lambda)$ and conclude that
$F^\infty ({\bf B}_\Lambda)= \overline{\cP(\{S_{i,s}\})}^{\text{\rm w*}}$.
The proof is complete.
\end{proof}

\begin{corollary} The noncommutative Hardy space $F^\infty({\bf B}_\Lambda)$ is a Banach algebra.
\end{corollary}

\begin{theorem}   \label{F}  Let $A\in F^\infty({\bf B}_\Lambda)$ have a formal Fourier representation
$$
\varphi(\{S_{i,s}\})=\sum_{(\beta_1,\ldots, \beta_k)\in  \FF_{n_1}^+\times\cdots \times \FF_{n_k}^+} c_{(\beta_1,\ldots, \beta_k)} S_{1,\beta_1}\ldots S_{k,\beta_k}.
$$
Then $\varphi(\{rS_{i,s}\})\in \cA({\bf B}_\Lambda)$, for all $r\in [0,1)$,
$$
A=\text{\rm SOT-}\lim_{r\to1}\varphi(\{rS_{i,s}\})
$$
and
$$
\|A\|=\sup_{0\leq r<1}\|\varphi(\{rS_{i,s}\})\|=\lim_{r\to 1}\|\varphi(\{rS_{i,s}\})\|.
$$
\end{theorem}
\begin{proof}
Since $\Phi_{S_i}$ is a completely positive linear map with $\|\Phi_{S_i}(I)\|\leq 1$, we have
$$
\Phi_{S_1}^{p_1}\circ \cdots \circ \Phi_{S_k}^{p_k}(I)
\leq
\|\Phi_{S_k}^{p_k}(I)\|\cdots \|\Phi_{S_1}^{p_1}(I)\| I
\leq
\|\Phi_{S_k}(I)\|^{p_k}\cdots \|\Phi_{S_1}(I)\|^{p_1} I\leq I
$$
for  all $p_1,\ldots p_k\in \NN$. Consequently, for every $r\in [0,1)$, we have
\begin{equation*}
\begin{split}
&\sum_{p=0}r^p\left\|\sum_{{p_1,\ldots, p_k\in \NN\cup\{0\}}\atop{p_1+\cdots +p_k=p}}
\sum_{{\beta_1\in \FF_{n_1}^+,\ldots \beta_k\in \FF_{n_k}^+}\atop
{|\beta_1|=p_1,\ldots, |\beta_k|=p_k}} c_{(\beta_1,\ldots, \beta_k)} S_{1,\beta_1}\ldots S_{k,\beta_k}\right\|\\
&\qquad\leq
\sum_{p=0}r^p\sum_{{p_1,\ldots, p_k\in \NN\cup\{0\}}\atop{p_1+\cdots +p_k=p}}
\left(\sum_{{\beta_1\in \FF_{n_1}^+,\ldots \beta_k\in \FF_{n_k}^+}\atop
{|\beta_1|=p_1,\ldots, |\beta_k|=p_k}}|c_{(\beta_1,\ldots, \beta_k)}|^2\right)^{1/2}
\|\Phi_{S_1}^{p_1}\circ \cdots \circ \Phi_{S_k}^{p_k}(I)\|^{1/2}\| \\
&\qquad
\leq\sum_{p=0}r^p\sum_{{p_1,\ldots, p_k\in \NN\cup\{0\}}\atop{p_1+\cdots +p_k=p}}
\left(\sum_{{\beta_1\in \FF_{n_1}^+,\ldots \beta_k\in \FF_{n_k}^+}\atop
{|\beta_1|=p_1,\ldots, |\beta_k|=p_k}}|c_{(\beta_1,\ldots, \beta_k)}|^2\right)^{1/2}\\
&\qquad = \left(\sum_{\beta_1\in \FF_{n_1}^+,\ldots \beta_k\in \FF_{n_k}^+} |c_{(\beta_1,\ldots, \beta_k)}|^2\right)^{1/2}
\left(\sum_{p=0}r^p\sum_{{p_1,\ldots, p_k\in \NN\cup\{0\}}\atop{p_1+\cdots +p_k=p}} 1\right)\\
&\qquad =
\left(\sum_{\beta_1\in \FF_{n_1}^+,\ldots \beta_k\in \FF_{n_k}^+} |c_{(\beta_1,\ldots, \beta_k)}|^2\right)^{1/2} \sum_{p=0}^\infty r^p \left(\begin{matrix} p+k-1\\k-1\end{matrix}\right)<\infty.
\end{split}
\end{equation*}
This shows that
$$
\varphi(\{rS_{i,s}\}):=
\sum_{p=0}^\infty \sum_{{(\beta_1,\ldots, \beta_k)\in  \FF_{n_1}^+\times\cdots \times \FF_{n_k}^+}\atop {|\beta_1|+\cdots +|\beta_k|=p}} r^{|\beta_1|+\cdots +|\beta_k|} c_{(\beta_1,\ldots, \beta_k)} S_{1,\beta_1}\ldots S_{k,\beta_k}
$$
converges in the operator norm topology and, consequently,
 $\varphi(\{rS_{i,s}\})\in \cA({\bf B}_\Lambda)$.

The next step is to show that
\begin{equation}
\label{bound-vN}
\|\varphi(\{rS_{i,s}\})\|\leq \|A\|,\qquad r\in [0,1).
\end{equation}
For each $n\in \NN$, set
$$
q_n(\{S_{i,s}\}):=\sum_{p=0}^n\sum_{{(\beta_1,\ldots, \beta_k)\in  \FF_{n_1}^+\times\cdots \times \FF_{n_k}^+}\atop {|\beta_1|+\cdots +|\beta_k|=p}}  c_{(\beta_1,\ldots, \beta_k)} S_{1,\beta_1}\ldots S_{k,\beta_k}
$$
and note that
$$
\varphi(\{rS_{i,s}\})^*\chi_{(\alpha_1,\ldots \alpha_k)}=q_n(\{rS_{i,s}\})^*\chi_{(\alpha_1,\ldots \alpha_k)}, \qquad r\in [0,1),
$$ and
$$A^*\chi_{(\alpha_1,\ldots \alpha_k)}=q_n(\{S_{i,s}\})^*\chi_{(\alpha_1,\ldots \alpha_k)}$$
for all $(\alpha_1,\ldots, \alpha_k)\in  \FF_{n_1}^+\times\cdots \times \FF_{n_k}^+$ with
$|\alpha_1|+\cdots +|\alpha_k|\leq  n$.
According to Theorem \ref{Berezin}, the noncommutative Berezin transform
$K_{rS}:\ell^2( \FF_{n_1}^+\times\cdots \times \FF_{n_k}^+)\to \ell^2( \FF_{n_1}^+\times\cdots \times \FF_{n_k}^+)\otimes \ell^2( \FF_{n_1}^+\times\cdots \times \FF_{n_k}^+)
$
satisfies the relation
$K_{rS}(rS_{i,s}^*)=(S_{i,s}^*\otimes I)K_{rS}$ for every $i\in \{1,\ldots, k\}$ and $s\in \{1,\ldots, n_i\}$.
Let $\boldsymbol \gamma:=(\gamma_1,\ldots, \gamma_k)$, $\boldsymbol \sigma:=(\sigma_1,\ldots, \sigma_k)$, and $\boldsymbol \omega:=(\omega_1,\ldots, \omega_k)$ be in $\FF_{n_1}^+\times\cdots \times \FF_{n_k}^+$. Due to the definition of $S_{i,s}$ we have
$S_{k,\beta_k}^*\cdots S_{1,\beta_1}^*
\chi_{\boldsymbol \gamma}=0$ if $|\beta_1|+\cdots +|\beta_k|>|\gamma_1|+\cdots  +|\gamma_k|$.
Using the relations above  and the definition of  $K_{rS}$ and taking $n\geq |\gamma_1|+\cdots +|\gamma_k|$, we obtain
\begin{equation*}
\begin{split}
&\left< K_{rS}\varphi(\{rS_{i,s}\})^*\chi_{\boldsymbol \gamma}, \chi_{\boldsymbol\sigma}\otimes \chi_{\boldsymbol \omega}\right>\\
&=
\left< K_{rS}q_n(\{rS_{i,s}\})^*\chi_{\boldsymbol \gamma}, \chi_{\boldsymbol\sigma}\otimes \chi_{\boldsymbol \omega}\right>\\
&=
\left< (q_n(\{S_{i,s}\})^*\otimes I)K_{rS}\chi_{\boldsymbol \gamma}, \chi_{\boldsymbol\sigma}\otimes \chi_{\boldsymbol \omega}\right>\\
&=
\left< (q_n(\{S_{i,s}\})^*\otimes I)\left(
\sum_{\beta_1\in \FF_{n_1}^+,\ldots, \beta_k\in \FF_{n_k}^+}
\chi_{(\beta_1,\ldots, \beta_k)}\otimes  r^{|\beta_1+\cdots |\beta_k|} \Delta_{rS}(I)^{1/2}S_{k,\beta_k}^*\cdots S_{1,\beta_1}^*
\chi_{\boldsymbol \gamma}\right), \chi_{\boldsymbol\sigma}\otimes \chi_{\boldsymbol \omega}\right>\\
&=
\sum_{\boldsymbol \beta:=(\beta_1,\ldots, \beta_k)\in\ \FF_{n_1}^+\times \cdots \times \FF_{n_k}^+} r^{|\beta_1+\cdots + |\beta_k|}
\left<q_n(\{S_{i,s}\})^* \chi_{\boldsymbol \beta}, \chi_{\boldsymbol\sigma}
\right>
\left< S_{k,\beta_k}^*\cdots S_{1,\beta_1}^*
\chi_{\boldsymbol \gamma},  \Delta_{rS}(I)^{1/2} \chi_{\boldsymbol \omega}\right>\\
&=
\sum_{\boldsymbol \beta:=(\beta_1,\ldots, \beta_k)\in\ \FF_{n_1}^+\times \cdots \times \FF_{n_k}^+} r^{|\beta_1+\cdots + |\beta_k|}
\left<A^* \chi_{\boldsymbol \beta}, \chi_{\boldsymbol\sigma}
\right>
\left< S_{k,\beta_k}^*\cdots S_{1,\beta_1}^*
\chi_{\boldsymbol \gamma},  \Delta_{rS}(I)^{1/2} \chi_{\boldsymbol \omega}\right>\\
&=
\left< (A^*\otimes I)K_{rS}\chi_{\boldsymbol \gamma}, \chi_{\boldsymbol\sigma}\otimes \chi_{\boldsymbol \omega}\right>\\
\end{split}
\end{equation*}
for all $r\in [0,1)$. Since $A$ and $\varphi(\{rS_{i,s}\})$ are bounded operators, we deduce that
 $$
 K_{rS}\varphi(\{rS_{i,s}\})^*=(A^*\otimes I)K_{rS},\qquad r\in [0,1).
 $$
Since $K_{rS}$ is an isometry, we have
$\varphi(\{rS_{i,s}\})=K_{rS}^*(A\otimes I)K_{rS}$ and
\begin{equation}
\label{ine-vn}
\|\varphi(\{rS_{i,s}\})\|\leq \|A\|, \qquad r\in [0,1),
\end{equation}
which proves relation \eqref{bound-vN}.
Consequently, taking into account that
$$
A \chi_{\boldsymbol\alpha}=\lim_{ r\to 1}\varphi(\{rS_{i,s}\}\chi_{\boldsymbol\alpha},\qquad
(\alpha_1,\ldots, \alpha_k)\in  \FF_{n_1}^+\times\cdots \times \FF_{n_k}^+,
$$
we conclude that  $A=\text{\rm SOT-}\lim_{r\to1}\varphi(\{rS_{i,s}\})$.

To prove the last part of the theorem, let $0<r_1<r_2<1$. Since
$\varphi(\{r_2S_{i,s}\})\in \cA({\bf B}_\Lambda)$, inequality \eqref{ine-vn} applied  to $A=\varphi(\{r_2S_{i,s}\})$ implies
$
\|\varphi(\{rr_2S_{i,s}\})\|\leq \|\varphi(\{r_2S_{i,s}\})|$ for any $r\in [0,1)$. Taking $r=\frac{r_1}{r_2}$, we deduce that
$
\|\varphi(\{r_1S_{i,s}\})\|\leq \|\varphi(\{r_2S_{i,s}\})|$. The rest of the proof is straightforward.
 \end{proof}

\bigskip
\section{Functional calculus
}

In this section, we prove the existence of an $F^\infty({\bf B}_\Lambda)$-functional calculus  for the completely non-coisometric  (c.n.c.) elements   in the $\Lambda$-polyball.   This  extends the Sz.-Nagy--Foia\c s  the functional calculus  for c.n.c. contractions    and the functional calculus for c.n.c row contractions.

First, we consider the case of pure $k$-tuples in the regular $\Lambda$-polyball.

\begin{theorem}  \label{fiT} Let  $T=(T_1,\ldots, T_k)$ be a pure $k$-tuple in the regular $\Lambda$-polyball ${\bf B}_\Lambda(\cH)$,  where $\cH$ is a separable Hilbert space, and let
$\Psi_T: F^\infty({\bf B}_\Lambda) \to B(\cH)$  be defined by
$$
\Psi_T(A):=K_T^*(A\otimes I)K_T,\qquad A\in F^\infty({\bf B}_\Lambda),
$$
where $K_T$ is the noncommutative Berezin kernel associated with $T$.
Then the following statements hold.
\begin{enumerate}
\item[(i)]  $\Psi_T$  is WOT-(resp. SOT-) continuous  on bounded sets.

\item[(ii)]
$\Psi_T$ is a unital completely contractive homomorphism which is w*-continuous.
\item[(iii)]
If
$$
\varphi(\{S_{i,s}\})=\sum_{(\beta_1,\ldots, \beta_k)\in  \FF_{n_1}^+\times\cdots \times \FF_{n_k}^+} c_{(\beta_1,\ldots, \beta_k)} S_{1,\beta_1}\ldots S_{1,\beta_k}
$$
is the formal Fourier representation of $A\in F^\infty({\bf B}_\Lambda)$, then
$$
\Psi_T(A)=\text{\rm SOT-}\lim_{r\to 1} \varphi(\{rT_{i,s}\})
$$
and  $\Psi_T(p(\{S_{i,s}\}))=p(\{T_{i,s}\})$ for any polynomial $p(\{S_{i,s}\})\in \cP(\{S_{i,s}\})$.
\end{enumerate}

\end{theorem}

\begin{proof} Let $\{A_\iota\}_\iota$ be a bounded net in $F^\infty({\bf B}_\Lambda)$. Then  WOT-$\lim_{\iota}A_\iota=0$ if and only if w*-$\lim_{\iota}A_\iota=0$. The latter relation implies
WOT-$\lim_{\iota}A_\iota\otimes I_\cH=0$  and w*-$\lim_{\iota}A_\iota\otimes I_\cH=0$.
Now, it is clear that  WOT-$\lim_{\iota}K_T^*(A_\iota\otimes I_\cH)K_T=0$, thus  $\Psi_T$ is WOT-continuous. Since the map $A\mapsto A\otimes I_\cH$  is SOT-continous  on bounded sets, so is  $\Psi_T$.

To prove (ii), note first that a net  $\{A_\iota\}_\iota$ in $F^\infty({\bf B}_\Lambda)$  converges to 0 in the w*-topology if and only if $A_\iota\otimes I_\cH\to 0$ in the w*-topology. This implies that $\Psi_T$ is continuous in the w*-topology.

On the other hand, since $T$ is a pure $k$-tuple, the noncommutative Berezin kernel $K_T$ is an isometry. Due to Theorem \ref{Berezin}, we have
$$
\left[ \Psi_T(A_{ij}\right]_{m\times m}=\text{\rm diag}_m (K_T^*)\left[A_{ij}\otimes I\right]_{m\times m}\text{\rm diag}_m (K_T)
$$
which implies
$$
\left\|\left[ \Psi_T(A_{ij}\right]_{m\times m}\right\|\leq \left\|\left[A_{ij}\right]_{m\times m}\right\|
$$
for every matrix $\left[A_{ij}\otimes I\right]_{m\times m}$ with entries in $F^\infty({\bf B}_\Lambda)$.
This proves that $\Psi_T$ is a unital completely contractive linear map.

Due to Theorem \ref{Berezin}, $\Psi_T$ is a homomorphism on the algebra of polynomial
$\cP(\{S_{i,s}\})$ which, due to Theorem  \ref{densities},  is sequenytially WOT-dense in
$F^\infty({\bf B}_\Lambda)$.  Since $\Psi_T$  is WOT- continuous  on bounded sets and using the principle of uniform boundedness, one can easily see that $\Psi_T$  is a homomorphism on
$F^\infty({\bf B}_\Lambda)$. This completes the proof of item (ii).

Now, we prove part (iii) of the theorem. According to Theorem \ref{F}, we have
$$
A=\text{\rm SOT-}\lim_{r\to1}\varphi(\{rS_{i,s}\})\quad \text{ and } \quad \|A\|=\sup_{0\leq r<1}\|\varphi(\{rS_{i,s}\})\|.
$$
 Since the map $X\mapsto X\otimes I_\cH$ is SOT-continuous on bounded sets, we have
 \begin{equation}
 \label{KT}
 K_T^*(A\otimes I_\cH)K_T=\text{\rm SOT-}\lim_{r\to 1} K_T^*(\varphi(\{rS_{i,s}\})\otimes I_\cH)K_T.
 \end{equation}
 On the other hand,
 $$
\varphi(\{rS_{i,s}\})=\sum_{(\beta_1,\ldots, \beta_k)\in  \FF_{n_1}^+\times\cdots \times \FF_{n_k}^+} c_{(\beta_1,\ldots, \beta_k)} r^{|\beta_1|+\cdots +|\beta_k|}S_{1,\beta_1}\ldots S_{1,\beta_k}
$$
is in $\cA({\bf B}_\Lambda)$ and the convergence is in the operator norm.
Setting
$$
q_n(\{rS_{i,s}\}):=\sum_{{(\beta_1,\ldots, \beta_k)\in  \FF_{n_1}^+\times\cdots \times \FF_{n_k}^+}\atop{|\beta_1|+\cdots +|\beta_k|\leq n}} c_{(\beta_1,\ldots, \beta_k)} r^{|\beta_1|+\cdots +|\beta_k|}S_{1,\beta_1}\ldots S_{1,\beta_k},
$$
we have  $\varphi(\{rS_{i,s}\})=\lim_{n\to\infty}q_n(\{rS_{i,s}\})$. Using the von Neumann type inequality
\begin{equation}\label{Cauchy}
\|q_n(\{rT_{i,s}\})-q_m(\{rT_{i,s}\})\|\leq \|q_n(\{rS_{i,s}\})-q_m(\{rS_{i,s}\}\|,
\end{equation}
 we also deduce that
$\varphi(\{rT_{i,s}\})=\lim_{n\to\infty}q_n(\{rT_{i,s}\})$ in the norm topology. Consequently,
$$
K_T^*(\varphi(\{rS_{i,s}\})\otimes I_\cH)K_T=\lim_{n\to\infty}K_T^*(q_n(\{rS_{i,s}\})\otimes I_\cH)K_T
=\lim_{n\to\infty} q_n(\{rT_{i,s}\})=\varphi(\{rT_{i,s}\}).
$$
Hence, and using relation \eqref{KT}, we obtain
$$
\Psi_T(A)=K_T^*(A\otimes I_\cH)K_T=\text{\rm SOT-}\lim_{r\to 1}\varphi(\{rT_{i,s}\}).
$$
The fact that $\Psi_T(p(\{S_{i,s}\}))=p(\{T_{i,s}\})$ for any polynomial $p(\{S_{i,s}\})\in \cP(\{S_{i,s}\})$ is due to Theorem \ref{Berezin}.
The proof is complete.
\end{proof}

\begin{lemma} \label{Ber}
Let $T:=(T_1,\ldots, T_k)\in {\bf B}_\Lambda(\cH)$ and let $A\in F^\infty({\bf B}_\Lambda)$ have the Fourier representation
$$
\varphi(\{S_{i,s}\}):=\sum_{p=0}^\infty\sum_{{(\beta_1,\ldots, \beta_k)\in  \FF_{n_1}^+\times\cdots \times \FF_{n_k}^+}\atop {|\beta_1|+\cdots +|\beta_k|=p}}  c_{(\beta_1,\ldots, \beta_k)} S_{1,\beta_1}\ldots S_{k,\beta_k}.
$$
Then the series defining $\varphi(\{rT_{i,s}\}$, $r\in [0,1)$,  is convergent in the operator norm topology and
$$
\varphi(\{rT_{i,s}\}=K_{rT}^*(A\otimes I_\cH)K_{rT},\qquad r\in [0,1),
$$
where $K_T$ is the noncommutative Berezin kernel of $T$.
\end{lemma}
\begin{proof} The fact that the series defining $\varphi(\{rT_{i,s}\}$, $r\in [0,1)$,  is convergent in the operator norm topology  follows from
 the proof of Theorem \ref{fiT}, where we showed that
 $\varphi(\{rT_{i,s}\})=\lim_{n\to\infty}q_n(\{rT_{i,s}\})$. Moreover, if $\epsilon>0$, there exists $N\in \NN$ such that
   $$
 \|q_n(\{rtS_{i,s}\})-\varphi(\{rtT_{i,s}\})\|\leq  \|q_n(\{rS_{i,s}\})-\varphi(\{rT_{i,s}\})\|<\frac{\epsilon}{3}
$$
for every $t\in [0,1]$ and $n\geq N$. Let $\delta\in (0,1)$  be such that
$$
\|q_N(\{rtS_{i,s}\})-\|q_N(\{rS_{i,s}\}))\|<\frac{\epsilon}{3}, \qquad t\in [\delta,1).
$$
Now, we can see that
\begin{equation*}
\begin{split}
\|\varphi(\{rS_{i,s}\})-\varphi(\{rtS_{i,s}\})\|&\leq \|\varphi(\{rS_{i,s}\})-q_N(\{rS_{i,s}\})\|+
\|q_N(\{rS_{i,s}\})-q_N(\{rtS_{i,s}\})\|\\
&=\|q_N(\{rtS_{i,s}\})-\varphi(\{rtS_{i,s}\})\|<\epsilon
\end{split}
\end{equation*}
for every $t\in [\delta,1)$. This shows that
$\varphi(\{rS_{i,s}\})=\lim_{t\to1}\varphi(\{rtS_{i,s}\})$ in the operator norm.
On the other hand, as we saw in the proof of Theorem \ref{fiT},
$$
\varphi(\{rtS_{i,s}\})=K_{rT}^*(\varphi(\{tS_{i,s}\})\otimes I_\cH)K_{rT}, \qquad r,t\in [0,1).
$$
Using the fact that $X\mapsto X\otimes I_\cH$ is SOT-continuous on bounded sets
and, due to Theorem \ref{F},
$A=\text{\rm SOT-}\lim_{t\to1}\varphi(\{tS_{i,s}\})$, we pass to the limit in the relation above as $t\to 1$ and  obtain
$$
\varphi(\{rT_{i,s}\}=K_{rT}^*(A\otimes I_\cH)K_{rT},\qquad r\in [0,1),
$$
The proof is complete.
\end{proof}

We say that  $T:=(T_1,\ldots, T_k)\in {\bf B}_\Lambda(\cH)$ is  a completely non-coisometric  $k$-tuple if there is no $h\in \cH$ , $h\neq 0$, such that
$$
\left<(id-\Phi_{T_k}^{p_k})\circ\cdots \circ (id-\Phi_{T_1}^{p_1})(I)h,h\right>=0
$$
for all $(p_1,\ldots, p_k)\in \NN^k$. We saw in the proof of Theorem \ref{Berezin} that
$$
(id-\Phi_{T_k}^{p_k})\circ\cdots \circ (id-\Phi_{T_1}^{p_1})(I)
=\sum_{s_1}^{p_1-1}\Phi_{T_1}\circ\cdots \circ \left(\sum_{s_k=1}^{p_k-1} \Phi_{T_k}\circ(\Delta_T(I))\right).
$$
This shows that the sequence $\left\{(id-\Phi_{T_k}^{p_k})\circ\cdots \circ (id-\Phi_{T_1}^{p_1})(I)\right\}_{(p_1,\ldots, p_k)\in \NN^k}$ is  increasing and, consequently,  $T$  is completely non-coisometric if and only if there is no $h\in \cH$ , $h\neq 0$, such that
$$
\lim_{p_k\to\infty}\ldots \left<\lim_{p_1\to\infty}(id-\Phi_{T_k}^{p_k})\circ\cdots \circ (id-\Phi_{T_1}^{p_1})(I)h,h\right>=0.
$$
Note that each pure $k$-tuple is completely non-coisometric.

The main result of this section is the following

\begin{theorem} Let $T:=(T_1,\ldots, T_k)\in {\bf B}_\Lambda(\cH)$ be a completely non-coisometric tuple. Then
$$
\Psi_T(A):=\text{\rm SOT-}\lim_{r\to 1} K_{rT}^*(A\otimes I_\cH)K_{rT},\qquad
A\in F^\infty({\bf B}_\Lambda),
$$
exists and defines a linear map $\Psi_T: F^\infty({\bf B}_\Lambda)\to B(\cH)$
with the following properties.
\begin{enumerate}
\item[(i)] If $\varphi(\{S_{i,s}\})$ is the Fourier representation of $A\in F^\infty({\bf B}_\Lambda)$, then
$$
\Psi_T(A)=\text{\rm SOT-}\lim_{r\to 1} \varphi(\{rT_{i,s}\}).
$$

\item[(ii)] $\Psi_T$ is WOT-(resp. SOT-, w*-) continuous on bounded sets.
\item[(iii)] $\Psi_T$ is a unital completely contractive homomorphism.
\end{enumerate}

\end{theorem}
\begin{proof}
Let $A\in F^\infty({\bf B}_\Lambda)$ have the Fourier representation
$$
\varphi(\{S_{i,s}\}):=\sum_{p=0}^\infty\sum_{{(\beta_1,\ldots, \beta_k)\in  \FF_{n_1}^+\times\cdots \times \FF_{n_k}^+}\atop {|\beta_1|+\cdots +|\beta_k|=p}}  c_{(\beta_1,\ldots, \beta_k)} S_{1,\beta_1}\ldots S_{k,\beta_k}.
$$
According to Theorem \ref{Berezin}, we have
$$
T_{i,s}K_T^*=K_T^*(S_{i,s}\otimes I_\cH)
$$
for all $i\in \{1,\ldots, k\}$ and $s\in \{1,\ldots, n_i\}$, where $K_T$ is the noncommutative Berezin kernel of $T$.
Since the series
$\varphi(\{rS_{i,s}\})$, $r\in [0,1)$, is convergent in the operator norm, so is $\varphi(\{rT_{i,s}\})$. To see this, it is enough to use relation \eqref{Cauchy}, where
$$
q_n(\{rS_{i,s}\}):=\sum_{{(\beta_1,\ldots, \beta_k)\in  \FF_{n_1}^+\times\cdots \times \FF_{n_k}^+}\atop{|\beta_1|+\cdots +|\beta_k|\leq n}} c_{(\beta_1,\ldots, \beta_k)} r^{|\beta_1|+\cdots +|\beta_k|}S_{1,\beta_1}\ldots S_{1,\beta_k}.
$$
Now, note that
$$
q_n(\{rT_{i,s}\})K_T^*=K_T^*(q_n(\{rS_{i,s}\})\otimes I_\cH).
$$
Taking $n\to \infty$, we deduce that
\begin{equation}
\label{fik}
\varphi(\{rT_{i,s}\})K_T^*=K_T^*(\varphi(\{rS_{i,s}\})\otimes I_\cH).
\end{equation}
On the other hand,  due to Theorem \ref{F}, we have
$$
A\otimes I_\cH=\text{\rm SOT-}\lim_{r\to1}\varphi(\{rS_{i,s}\})\otimes I_\cH.
$$
Using the later relation in \eqref{fik}, we deduce that the map $\Omega:\text{\rm range}\, K_T^*\to \cH$ defined by
$$
\Omega(K_T^*f):=\lim_{r\to 1} \varphi(\{rT_{i,s}\}) K_T^* f,\qquad f\in \ell^2(\FF_{n_1}^+\times\cdots \times \FF_{n_k}^+)\otimes \cD_T,
$$
is well-defined, linear, and
\begin{equation*}
\begin{split}
\|\Omega K_T^*f\|&\leq \limsup_{r\to 1} \|\varphi(\{rT_{i,s}\})\|\|K_T^*f\|\\
&\leq   \limsup_{r\to 1} \|\varphi(\{rS_{i,s}\})\|\|K_T^*f\|\\
&\leq \|A\| \|K_T^*f\|.
\end{split}
\end{equation*}
Since $T$ is a completely non-coisometric tuple, Theorem \ref{Berezin} shows that $K_T^*K_T$ is a one-to-one operator, which implies $\overline{\text{\rm range}\,K_T^*}=\cH$.
Due to inequalities above, $\Omega$ has a unique extension $\widetilde \Omega$ to a  bounded operator on $\cH$ with $\|\widetilde \Omega\|\leq \|A\|$.

In what follows, we show that
\begin{equation}
\label{Om}
\widetilde \Omega h=\lim_{r\to 1} \varphi(\{rT_{i,s}\})h,\qquad h\in \cH.
\end{equation}
Fix $h\in \cH$ and let $\{h_k\}_{k=1}^\infty\subset \text{\rm range}\, K_T^*$ such that $h_k\to h$ as $k\to\infty$.
Since $\|\varphi(\{rT_{i,s}\})\|\leq \|\varphi(\{rS_{i,s}\})\|\leq \|A\|$ for every $r\in [0,1)$, we deduce that
\begin{equation*}
\begin{split}
\|\widetilde \Omega h-\varphi(\{rS_{i,s}\})h\|&\leq \|\widetilde \Omega h-\widetilde \Omega h_k\|
+\|\widetilde \Omega h_k -\varphi(\{rT_{i,s}\})h_k\|
  +\|\varphi(\{rT_{i,s}\})h_k-\varphi(\{rT_{i,s}\})h\|\\
&\leq \|\widetilde \Omega \|\|h-h_k\|+\|\widetilde \Omega h_k-\varphi(\{rT_{i,s}\})h_k\|
+\|\varphi(\{rT_{i,s}\})\|\|h_k-h\| \\
&\leq 2\|A\|\|h-h_k\| +\|\widetilde \Omega h_k-\varphi(\{rT_{i,s}\})h_k\|.
\end{split}
\end{equation*}
Using the fact that $\widetilde \Omega h_k-\lim_{r\to 1} \varphi(\{rT_{i,s}\})h_k$, we deduce relation \eqref{Om}. According to Lemma \ref{Ber}, we have
$$
\varphi(\{rT_{i,s}\}=K_{rT}^*(A\otimes I_\cH)K_{rT},\qquad r\in [0,1).
$$
Consequently, taking $r\to 1$ and using relation \eqref{Om}, we obtain
$$
 \widetilde \Omega =\text{\rm SOT-}\lim_{r\to 1} K_{rT}^*(A\otimes I_\cH)K_{rT},
 $$
which shows that  $\Psi_T(A)= \widetilde \Omega$. Therefore, item (i) holds.
To prove part (ii), let $[A_{pq}]_{m\times m}$ be a matrix with entries in
 $F^\infty({\bf B}_\Lambda)$ and let $\varphi_{pq}(\{S_{i,s}\}$ be the Fourier representation of $A_{pq}$. Lemma \ref{Ber} shows that
 $$
[\varphi_{pq}(\{rT_{i,s}\}]_{m\times m}=\text{\rm diag}_m(K_{rT}^*)[A_{pq}\otimes I_\cH]_{m\times m}\text{\rm diag}_m(K_{rT}),\qquad r\in [0,1).
$$
On the other hand, since $K_{rT}$ is an isometry, we deduce that
$$
\|[\varphi_{pq}(\{rT_{i,s}\}]_{m\times m}\|\leq \|[A_{pq}\otimes I_\cH]_{m\times m}\|,\qquad r\in [0,1), m\in\NN.
$$
Since $
\Psi_T(A_{pq})=\text{\rm SOT-}\lim_{r\to 1} \varphi_{pq}(\{rT_{i,s}\}),
$
we deduce that $\Psi_T$ is a completely contractive linear map. Now, using that fact that
 $\Psi_T$ is a homomorphism  on the algebra of polynomials $\cP(\{S_{i,s}\})$ and that
$F^\infty({\bf B}_\Lambda)$ is the sequential WOT-closure of $\cP(\{S_{i,s}\})$ (see Theorem
\ref{densities}), one can use the WOT-continuity of $\Psi_T$ on bounded sets to deduce that
$\Psi_T$ is a homomorphism on $F^\infty({\bf B}_\Lambda)$.

Now, we prove part (iii).  Due to the proof of part (i), we have $\|\Psi_T(A)\|\leq \|A\|$ for all
$A\in F^\infty({\bf B}_\Lambda)$. On the other hand,  taking $r\to 1$ in relation \eqref{fik} we obtain
\begin{equation}
\label{AA}
\Psi_T(A) K_T^*=K_T^*(A\otimes I_\cH),\qquad A\in F^\infty({\bf B}_\Lambda).
\end{equation}
Let $\{A_\iota\}$ be a bounded net in $F^\infty({\bf B}_\Lambda)$ such that  $A_\iota \to A\in F^\infty({\bf B}_\Lambda)$ in the WOT (resp. SOT). Then  $A_\iota\otimes I_\cH \to A\otimes I_\cH $ in the WOT (resp. SOT).  Due to relation \eqref{AA}, we have
$\Psi_T(A_\iota) K_T^*=K_T^*(A_\iota\otimes I_\cH)$.
Since $\overline{\text{\rm range} \,K_T^*}=\cH$ and $\{\Psi_T(A_\iota)\}_\iota$ is a bounded net, we can easily see that $\Psi_T(A_\iota)\to \Psi_T(A)$ in  the WOT (resp. SOT).
The proof is complete.
\end{proof}

\bigskip

\section{ Free holomorphic functions on regular $\Lambda$-polyballs}

In this section,   we introduce the algebra  $H^\infty( {\bf B}_\Lambda^\circ )$ of  bounded  free holomorphic functions  on the interior of ${\bf B}_\Lambda(\cH)$, for any Hilbert space $\cH$,  and prove that it  is  completely isometric isomorphic to the noncommutative Hardy algebra  $F^\infty({\bf B}_\Lambda)$ introduced in Section 2. We also introduce the  albegra
 $ A( {\bf B}_\Lambda^\circ)$  and show that it is completely isometric isomorphic to the noncommutative  
 $\Lambda$-polyball algebra  $\cA({\bf B}_\Lambda)$.

If $A\in B(\cH)$ is an invertible positive operator, we write $A>0$. Recall that if $X\in {\bf B}_\Lambda(\cH)$, then
$$
\Delta_X(I):=(id-\Phi_{X_k})\circ\cdots \circ (id-\Phi_{X_1})(I). 
$$
\begin{proposition} \label{interior}
The set 
$${\bf B}_\Lambda^\circ (\cH):=\{ X\in {\bf B}_\Lambda (\cH): \ \Delta_X(I)>0\}
$$
  is relatively open in ${\bf B}_\Lambda (\cH)$ and
$$\overline{{\bf B}_\Lambda^\circ (\cH)}={\bf B}_\Lambda (\cH).
$$
Moreover, the interior
of ${\bf B}_\Lambda (\cH)$  coincides with ${\bf B}_\Lambda^\circ (\cH)$.
\end{proposition}
\begin{proof}
Let $X=(X_1,\ldots, X_k)\in {\bf B}_\Lambda^\circ (\cH)$ and assume that
$\Delta_X(I)> cI$ for some $c>0$. If $d\in (0,c)$, then there exists $\epsilon>0$ such that for all $Y=(Y_1,\ldots Y_k)\in {\bf B}_\Lambda (\cH)$ with $\|X_i-Y_i\|<\epsilon$ for $i\in \{1,\ldots, k\}$, we have
$$
-dI\leq \Delta_Y(I)-\Delta_X(I)\leq dI.
$$
Hence,
$$
\Delta_Y(I)=(\Delta_Y(I)-\Delta_X(I))+\Delta_X(I)\geq (c-d)I>0
$$
and, consequently,  $Y\in {\bf B}_\Lambda^\circ (\cH)$. Therefore, ${\bf B}_\Lambda^\circ (\cH)$ is a relatively open set in ${\bf B}_\Lambda (\cH)$.

Now, we prove that $\overline{{\bf B}_\Lambda^\circ (\cH)}={\bf B}_\Lambda (\cH)$.
To prove the  inclusion
 $ \overline{{\bf B}_\Lambda^\circ (\cH)}\subset{\bf B}_\Lambda (\cH)$, let
 $Y=(Y_1,\ldots Y_k)\in \overline{{\bf B}_\Lambda (\cH)}$, and  let $Y^{(n)}=(Y_1^{(n)},\ldots Y_k^{(n)})\in  {\bf B}_\Lambda (\cH)$ be  a sequence such that $Y^{(n)}\to Y$, as $n\to \infty$, in the norm topology of $B(\cH)^{n_1+\cdots +n_k}$.  Since,
for every $i,j\in \{1,\ldots, k\}$ with $i\neq j$ and every $s\in \{1,\ldots, n_i\}$, $t\in \{1,\ldots, n_j\}$,
 $$ Y^{(n)}_{i,s} Y^{(n)}_{j,t}=\lambda_{ij}(s,t)Y^{(n)}_{j,t}Y^{(n)}_{i,s},
$$
taking $n\to \infty$, we obtain $ Y_{i,s} Y_{j,t}=\lambda_{ij}(s,t)Y_{j,t}Y_{i,s}.
$
On the other hand,  we have
$$
(id-\Phi_{rY^{(n)}_k})\circ\cdots \circ (id-\Phi_{rY^{(n)}_1})(I)\geq 0,\qquad r\in [0,1), n\in \NN,
$$
which implies
$$
(id-\Phi_{rY_k})\circ\cdots \circ (id-\Phi_{rY_1})(I)\geq 0,\qquad r\in [0,1).
$$
Consequently, $Y\in {\bf B}_\Lambda(\cH)$.

Now, we prove the  inclusion
 ${\bf B}_\Lambda (\cH)\subset \overline{{\bf B}_\Lambda^\circ (\cH)}$.
Let $Y\in {\bf B}_\Lambda (\cH)$ and $r\in [0,1)$. Using Lemma 4.3 from \cite{Po-twisted}, we deduce that 
$$
(id-\Phi_{rS_k})\circ\cdots \circ (id-\Phi_{rS_1})(I)=\prod_{i=1}^k (I-\Phi_{rS_i}(I))\geq \prod_{i=1}^k (1-r^2)I.
$$
Applying Theorem \ref{Berezin} when $X=tY$, $t\in [0,1)$, we obtain
\begin{equation*}
\begin{split}
(id-\Phi_{rtY_k})\circ\cdots \circ (id-\Phi_{rtY_1})(I)
&=K_{tY}^*\left[ (id-\Phi_{rS_k})\circ\cdots \circ (id-\Phi_{rS_1})(I)\right]K_{tY}\\
&\geq  \prod_{i=1}^k (1-r^2)I.
\end{split}
\end{equation*}
Here, we use the fact that $tY$ is a pure $k$--tuple and $K_{tY}$ is an isometry.
Taking $t\to 1$, we get
$$
(id-\Phi_{rY_k})\circ\cdots \circ (id-\Phi_{rY_1})(I) \geq  \prod_{i=1}^k (1-r^2)I
$$
which shows that $rY\in {\bf B}_\Lambda^\circ (\cH)$ for all $r\in [0,1)$. Hence, it is clear that
$Y\in \overline{{\bf B}_\Lambda^\circ (\cH)}$.

Now, we prove the last part of the proposition.
If $X\in Int({\bf B}_\Lambda(\cH))$,  the interior of  ${\bf B}_\Lambda(\cH)$, then there exists $r_0\in (0,1)$ such that $\frac{1}{r_0}X\in {\bf B}_\Lambda(\cH)$. Hence, $X\in r_0{\bf B}_\Lambda(\cH)$. Thus $X=r_0Y$ for some $Y\in {\bf B}_\Lambda(\cH)$. We proved above that $r_0Y\in {\bf B}_\Lambda^\circ (\cH)$.
Consequently, $Int({\bf B}_\Lambda(\cH))\subset {\bf B}_\Lambda^\circ (\cH)$.
Since ${\bf B}_\Lambda^\circ (\cH)$ is relatively open in ${\bf B}_\Lambda (\cH)$, we conclude that
$Int({\bf B}_\Lambda(\cH))={\bf B}_\Lambda^\circ (\cH)$.
The proof is complete.
\end{proof}
\begin{corollary}
${\bf B}_\Lambda^\circ (\cH)=\bigcup_{0\leq r<1}r{\bf B}_\Lambda (\cH)$.
\end{corollary}

For each $i\in \{1,\ldots, k\}$, let $Z_i=(Z_{i,1},\ldots, Z_{i,n_i})$ be an $n_i$-tuple of noncommutative indeterminates subject to the relations
$$ Z_{i,s} Z_{j,t}=\lambda_{ij}(s,t)Z_{j,t}Z_{i,s}
$$
for every $i,j\in \{1,\ldots, k\}$ with $i\neq j$ and every $s\in \{1,\ldots, n_i\}$, $t\in \{1,\ldots, n_j\}$.
We set  $Z_{i,\alpha}:=Z_{i,p_1}\cdots Z_{i,p_m}$ if $\alpha=g_{p_1}^i\cdots g_{p_m}^i\in \FF_{n_i}^+$, where $p_1,\ldots, p_m\in \{1,\ldots, n_i\}$  and  $Z_{i,g_0^i}:=1$.
If $\boldsymbol \beta:=(\beta_1,\ldots, \beta_k)\in \FF_{n_1}^+\times \cdots \times \FF_{n_k}^+$, we denote $Z_{\boldsymbol \beta}:=Z_{1,\beta_1}\cdots Z_{k,\beta_k}$ and $a_{\boldsymbol \beta}:=a_{(\beta_1\ldots\beta_k)}\in \CC$.
A formal power series
$$\varphi:=\sum_{\boldsymbol \beta\in \FF_{n_1}^+\times \cdots \times \FF_{n_k}^+} a_{\boldsymbol \beta} Z_{\boldsymbol \beta},\qquad a_{\boldsymbol \beta}\in \CC,
$$
in indeterminates $Z_{i,s}$, where $i\in \{1,\ldots, k\}$ and $s\in \{1,\ldots, n_i\}$, is called free holomorphic function on ${\bf B}_\Lambda^\circ$ if the series
$$\varphi(\{X_{i,s}\}):=\sum_{p=0}^\infty\sum_{{\boldsymbol\beta=(\beta_1,\ldots, \beta_k)\in  \FF_{n_1}^+\times\cdots \times \FF_{n_k}^+}\atop {|\beta_1|+\cdots +|\beta_k|=p}} a_{\boldsymbol \beta} X_{\boldsymbol \beta}
$$
is convergent  in the operator norm topology for any $X\in  {\bf B}_\Lambda^\circ (\cH)$ and any Hilbert space $\cH$.
We remark that the coefficients of a free holomorphic functions on $ {\bf B}_\Lambda^\circ $ are uniquely determined by its representation on an infinite dimensional separable Hilbert space.
Indeed, assume that   $\varphi(\{rS_{i,s}\})=0$ for any $r\in [0,1)$, where ${\bf S}=({\bf S}_1,\ldots, {\bf S}_k)$ is  the universal model associated with the $\Lambda$-polyball $ {\bf B}_\Lambda$. Using relation  \eqref{shift}, we obtain
\begin{equation*}
\begin{split}
0&=\left<\varphi(\{rS_{i,s}\}) \chi_{{\bf g}_0}, S_{1,\alpha_1}\cdots S_{k,\alpha_k}\chi_{{\bf g}_0}\right>\\
&=r^{|\alpha_1|+\cdots + |\alpha_k|}
\left< \sum_{(\beta_1,\ldots, \beta_k)\in  \FF_{n_1}^+\times\cdots \times \FF_{n_k}^+} a_{(\beta_1,\ldots, \beta_k)} \boldsymbol \mu(\boldsymbol \beta, \bf{g}_0)\chi_{(\beta_1,\ldots, \beta_k)}, \boldsymbol \mu(\boldsymbol \alpha, \bf{g}_0)\chi_{(\alpha_1,\ldots, \alpha_k)}
\right>\\
&=r^{|\alpha_1|+\cdots + |\alpha_k|}a_{(\alpha_1,\ldots, \alpha_k)} |\boldsymbol \mu(\boldsymbol \alpha, {\bf g}_0)|^2=r^{|\alpha_1|+\cdots + |\alpha_k|}a_{(\alpha_1,\ldots, \alpha_k)}.
\end{split}
\end{equation*}
Hence, $a_{(\alpha_1,\ldots, \alpha_k)}=0$, which proves our assertion.
We denote by $Hol( {\bf B}_\Lambda^\circ )$ the set of all free holomorphic functions on $ {\bf B}_\Lambda^\circ $.

\begin{proposition}  Let  ${\bf S}=({\bf S}_1,\ldots,{\bf S}_k)$ be the universal model associated with the $\Lambda$-polyball $ {\bf B}_\Lambda$.  Then
$\varphi:=\sum_{\boldsymbol \beta\in \FF_{n_1}^+\times \cdots \times \FF_{n_k}^+} a_{\boldsymbol \beta} Z_{\boldsymbol \beta}
$
is in $Hol( {\bf B}_\Lambda^\circ )$ if and only if  the series
$$
\varphi(\{rS_{i,s}\}):=\sum_{p=0}^\infty\sum_{{\boldsymbol\beta=(\beta_1,\ldots, \beta_k)\in  \FF_{n_1}^+\times\cdots \times \FF_{n_k}^+}\atop {|\beta_1|+\cdots +|\beta_k|=p}}  r^{|\beta_1|+\cdots +|\beta_k|}a_{\boldsymbol \beta} {\bf S}_{\boldsymbol \beta}
$$
is convergent in the operator norm topology for all $r\in [0,1)$.
\end{proposition}
\begin{proof} The direct implication is obvious. Note that the converse of the proposition is due to the noncommutative von Neumann inequality.
\end{proof}
We remark  that  $Hol( {\bf B}_\Lambda^\circ )$ is an algebra.  Let $H^\infty( {\bf B}_\Lambda^\circ )$ be the set of all $\varphi\in Hol( {\bf B}_\Lambda^\circ )$
such that
$$
\|\varphi\|_\infty:=\sup\|\varphi(\{X_{i,s}\})\|<\infty,
$$
where the supremum is taken over all $\{X_{i,s}\}\in {\bf B}_\Lambda^\circ(\cH)$ and any Hilbert space $\cH$. It is easy to see that $H^\infty( {\bf B}_\Lambda^\circ )$ is a Banach algebra under pointwise multiplication and the norm $\|\cdot\|_\infty$. There is an operator space structure on $H^\infty( {\bf B}_\Lambda^\circ )$, in the sense of Ruan \cite{P-book},  if we define the norms
$\|\cdot\|_m$ on $M_{m\times m}(H^\infty( {\bf B}_\Lambda^\circ )) $ by setting
$$
\|[\varphi_{uv}]_{m\times m}\|_m:=\sup\|[\varphi_{uv}(\{X_{i,s}\})]_{m\times m}\|,
$$
where the supremum is taken over all $\{X_{i,s}\}\in {\bf B}_\Lambda^\circ(\cH)$ and any Hilbert space.
 We remark that if $\varphi\in Hol( {\bf B}_\Lambda^\circ )$ and $r\in [0,1)$, then $\varphi$ is continuous on $r{\bf B}_\Lambda(\cH)$ and
 $$
 \|\varphi(\{X_{i,s}\})\|\leq \|\varphi(\{rS_{i,s}\})\|
$$
for every $\{X_{i,s}\}\in r {\bf B}_\Lambda(\cH)$. Moreover, the series defining $\varphi(\{X_{i,s}\})$ converges uniformly on $r {\bf B}_\Lambda(\cH)$ in the operator norm topology.

Given $A\in F^\infty({\bf B}_\Lambda)$ and a Hilbert space $\cH$, we define the {\it noncommutative Berezin transform} associated with the regular $\Lambda$-polyball
$ {\bf B}_\Lambda^\circ(\cH)$ to be the map
${\bf B}[A]: {\bf B}_\Lambda^\circ(\cH)\to B(\cH)$ defined by
$$
{\bf B}[A](X):={K}_X^*[A\otimes I_\cH]{K}_X,\qquad X\in {\bf B}_\Lambda^\circ(\cH).
$$

\begin{theorem}
\label{H-inf}
The map $\Gamma: H^\infty( {\bf B}_\Lambda^\circ )\to F^\infty({\bf B}_\Lambda)$
defined by
$$\Gamma\left(\sum_{\boldsymbol \beta\in \FF_{n_1}^+\times \cdots \times \FF_{n_k}^+} a_{\boldsymbol \beta} Z_{\boldsymbol \beta}\right):=\sum_{\boldsymbol \beta\in \FF_{n_1}^+\times \cdots \times \FF_{n_k}^+} a_{\boldsymbol \beta} {\bf S}_{\boldsymbol \beta}
$$
is a completely isometric isomorphism of operator algebras. Moreover, if  $f\in Hol( {\bf B}_\Lambda^\circ )$, then the following statements are equivalent.
\begin{enumerate}
\item[(i)]  $f\in H^\infty( {\bf B}_\Lambda^\circ )$;
\item[(ii)]  $\sup_{1\leq r<1}\|f(\{rS_{i,s}\})\|<\infty$;
\item[(iii)] there exists $A\in F^\infty({\bf B}_\Lambda)$ with $f={\bf B}[A]$, where ${\bf B}$ is the noncommutative Berezin transform associated with  the $\Lambda$-polyball ${\bf B}_\Lambda^\circ$.
\end{enumerate}
In this case, we have
$$\Gamma(f)=\text{\rm SOT-}\lim_{r\to 1} f(\{rS_{i,s}\}\quad \text{and}\quad \Gamma^{-1}(f)={\bf B}[A].
$$
Moreover,  $\|\Gamma(f)\|=\sup_{1\leq r<1}\|f(\{rS_{i,s}\})\|$.
\end{theorem}
\begin{proof} Let $f=\sum_{\boldsymbol \beta\in \FF_{n_1}^+\times \cdots \times \FF_{n_k}^+} a_{\boldsymbol \beta} Z_{\boldsymbol \beta}$ be in $H^\infty( {\bf B}_\Lambda^\circ )$. Since
$r{\bf S}\in {\bf B}_\Lambda^\circ (\ell^2 (\FF_{n_1}^+\times \cdots \times \FF_{n_k}^+))$ for all $r\in [0,1)$, the series
$$
f(\{rS_{i,s}\}):=\sum_{p=0}^\infty\sum_{{\boldsymbol\beta=(\beta_1,\ldots, \beta_k)\in  \FF_{n_1}^+\times\cdots \times \FF_{n_k}^+}\atop {|\beta_1|+\cdots +|\beta_k|=p}}  r^{|\beta_1|+\cdots +|\beta_k|}a_{\boldsymbol \beta} {\bf S}_{\boldsymbol \beta}
$$
is convergent in the operator norm topology for all $r\in [0,1)$ and
$M:=\sup_{1\leq r<1}\|f(\{rS_{i,s}\})\|<\infty$. Consequently, for every $r\in [0,1)$ and $\gamma\in
\FF_{n_1}^+\times \cdots \times \FF_{n_k}^+$, we have
$$
 f(\{rS_{i,s}\})(\chi_{\boldsymbol \gamma})=\sum_{(\beta_1,\ldots, \beta_k)\in  \FF_{n_1}^+\times\cdots \times \FF_{n_k}^+} a_{(\beta_1,\ldots, \beta_k)} r^{|\beta_1|+\cdots +|\beta_k|} \boldsymbol \mu(\boldsymbol \beta, \boldsymbol\gamma)\chi_{(\beta_1\gamma_1,\ldots, \beta_k\gamma_k)}
 $$
 and
 $$
 \sum_{(\beta_1,\ldots, \beta_k)\in  \FF_{n_1}^+\times\cdots \times \FF_{n_k}^+} |a_{(\beta_1,\ldots, \beta_k)}|^2 r^{2(|\beta_1|+\cdots +|\beta_k|)}=\|f(\{rS_{i,s}\})(\chi_{g_0})\|^2<M^2.
 $$
  Hence,  $\sum_{(\beta_1,\ldots, \beta_k)\in  \FF_{n_1}^+\times\cdots \times \FF_{n_k}^+} |a_{(\beta_1,\ldots, \beta_k)}|^2<M^2$ and, for every polynomial $p\in \cP$ in $\ell^2( \FF_{n_1}^+\times\cdots \times \FF_{n_k}^+)$, we have
  $f(\{rS_{i,s}\})p\to f(\{S_{i,s}\})p$ as $r\to 1$. Since  $\sup_{1\leq r<1}\|f(\{rS_{i,s}\})\|<\infty$, we deduce that  $\sup_{p\in \cP, \|p\|\leq 1} \|f(\{S_{i,s}\})p\|<\infty$.
Consequently, $\sum_{\boldsymbol \beta\in \FF_{n_1}^+\times \cdots \times \FF_{n_k}^+} a_{\boldsymbol \beta} S_{\boldsymbol \beta}$ is the Fourier series of an element   $A\in F^\infty({\bf B}_\Lambda)$ which, according to Theorem \ref{F}, satisfies the relation
$A=\text{\rm SOT-}\lim_{r\to 1} f(\{rS_{i,s}\}$ and $\|A\|=\sup_{1\leq r<1}\|f(\{rS_{i,s}\})\|$.
This proves that $\Gamma$ is a well-defined  isometric linear map. The fact that $\Gamma$ is surjective is due to Theorem \ref{F} and the fact that
 $
 \|\varphi(\{X_{i,s}\})\|\leq \|\varphi(\{rS_{i,s}\})\|
$
for any $\{X_{i,s}\}\in r {\bf B}_\Lambda(\cH)$. Passing to matrices, we can use similar techniques to show that  $\Gamma$ is a completely isometric isomorphism. The rest of the proof follows from   Theorem \ref{F} and Theorem \ref{fiT}. The proof is complete.
 \end{proof}

Denote by $A( {\bf B}_\Lambda^\circ)$ the set of all functions $f\in  Hol( {\bf B}_\Lambda^\circ )$ such that the map
${\bf B}_\Lambda^\circ (\cH)\ni X\mapsto f(X)\in B(\cH)$
has a continuous extension to ${\bf B}_\Lambda(\cH)$ for every Hilbert space $\cH$.
Using standards arguments, we can show that $A( {\bf B}_\Lambda^\circ)$ is a Banach algebra with pointwise multiplication and the norm $\|\cdot\|_\infty$.
It also has an operator space structure with respect to the norms $\|\cdot\|_m$, $m\in \NN$.
One can prove the following result.

\begin{theorem}
\label{A-inf}
The map $\Gamma: A( {\bf B}_\Lambda^\circ)\to \cA({\bf B}_\Lambda)$
defined by
$$\Gamma\left(\sum_{\boldsymbol \beta\in \FF_{n_1}^+\times \cdots \times \FF_{n_k}^+} a_{\boldsymbol \beta} Z_{\boldsymbol \beta}\right):=\sum_{\boldsymbol \beta\in \FF_{n_1}^+\times \cdots \times \FF_{n_k}^+} a_{\boldsymbol \beta} S_{\boldsymbol \beta}
$$
is a completely isometric isomorphism of operator algebras. Moreover, if  $f\in Hol( {\bf B}_\Lambda^\circ )$, then the following statements are equivalent.
\begin{enumerate}
\item[(i)]  $f\in A( {\bf B}_\Lambda^\circ )$;
\item[(ii)]  $\lim_{r\to 1}f(\{rS_{i,s}\})$ exists in the operator norm topology;
\item[(iii)] there exists $A\in \cA({\bf B}_\Lambda)$ with $f={\bf B}[A]$, where ${\bf B}$ is the noncommutative Berezin transform.
\end{enumerate}
In this case, we have
$$\Gamma(f)=\text{\rm SOT-}\lim_{r\to 1} f(\{rS_{i,s}\}\quad \text{and}\quad \Gamma^{-1}(f)={\bf B}[A].
$$
\end{theorem}
\begin{proof} Using Theorem \ref{H-inf}, Theorem 4.9  from \cite{Po-twisted}, and an approximation argument, one can complete the proof.
\end{proof}

\bigskip
\section{Characteristic functions and multi-analytic  models}
In this section, we characterize  the elements   in  the noncommutative $\Lambda$-polyball  which  admit a characteristic functions. We provide a model theorem for the class of    completely non-coisometric $k$-tuple  of operators  in  
${\bf B}_\Lambda(\cH)$ which  admit characteristic functions, and     show  that the  characteristic function
   is a complete unitary invariant for this class of $k$-tuples.

An operator $A:\ell^2(\FF_{n_1}^+\times\cdots \times \FF_{n_k}^+)\otimes \cH\to \ell^2(\FF_{n_1}^+\times\cdots \times \FF_{n_k}^+)\otimes \cK$ is called {\it multi-analytic}
 with respect to the universal model ${\bf S}=({\bf S}_1,\ldots, {\bf S}_k)$, ${\bf S}_i=(S_{i,1},\ldots, S_{i,n_i})$,  associated with the $\Lambda$-polyball ${\bf B}_\Lambda$
if
$$
A(S_{i,s}\otimes I_\cH)=(S_{i,s}\otimes I_\cK)A
$$
for every $i\in\{1,\ldots, k\}$ and  $s\in \{1,\ldots, n_i\}$. If, in addition,  $A$ is a partial isometry, we call it  an {\it inner multi-analytic} operator.
The support of $A$ is the smallest reducing subspace $\supp (A)\subset \ell^2(\FF_{n_1}^+\times\cdots \times \FF_{n_k}^+)\otimes \cH$ under all the operators $S_{i,s}$, containing the co-invariant subspace $\overline{A^*(\ell^2(\FF_{n_1}^+\times\cdots \times \FF_{n_k}^+)\otimes \cK)}$. According to Theorem 5.1 from \cite{Po-twisted}, we have
$$
\supp (A)=\ell^2(\FF_{n_1}^+\times\cdots \times \FF_{n_k}^+)\otimes \cL,
$$
where $\cL:=({\bf P}_\CC\otimes I_\cH)\overline{A^*(\ell^2(\FF_{n_1}^+\times\cdots \times \FF_{n_k}^+)\otimes \cK)}$ and  ${\bf P}_\CC$ is the orthogonal projection of
$\ell^2(\FF_{n_1}^+\times\cdots \times \FF_{n_k}^+)$ onto $\CC$ which is identified to the subspace $\CC\chi_{(g_1^0,\ldots g_k^0)}$ of
$\ell^2(\FF_{n_1}^+\times\cdots \times \FF_{n_k}^+)$.

In \cite{Po-twisted}, we proved the following Beurling type factorization result.

\begin{theorem} \label{Beurling-fac}
Let ${\bf S}=({\bf S}_1,\ldots, {\bf S}_k)$ be the universal model associated with  the $\Lambda$-polyball and let $Y$ be a selfadjoint operator on the Hilbert space $\ell^2(\FF_{n_1}^+\times\cdots \times \FF_{n_k}^+)\otimes \cK$. Then the following statements are equivalent.
\begin{enumerate}
\item[(i)] There is a multi-analytic operator $A: \ell^2(\FF_{n_1}^+\times\cdots \times \FF_{n_k}^+)\otimes \cL\to
\ell^2(\FF_{n_1}^+\times\cdots \times \FF_{n_k}^+)\otimes \cK$ such that
$$Y=AA^*.$$

\item[(ii)] $(id-\Phi_{{\bf S}_1\otimes I_\cK})\circ\cdots \circ (id-\Phi_{{\bf S}_k\otimes I_\cK})(Y)\geq 0$, where the completely positive maps $\Phi_{{\bf S}_i\otimes I_\cK}$ are defined in Section 1.
\end{enumerate}
\end{theorem}
We recall  \cite{Po-twisted} the construction of the operator $A$ in part (i) of Theorem \ref{Beurling-fac} .
Consider the subspace $\cG:=\overline{Y^{1/2}\left(\ell^2(\FF_{n_1}^+\times\cdots \times \FF_{n_k}^+)\otimes \cK\right)}$ and set
\begin{equation*}
C_{i,s}(Y^{1/2} g):=Y^{1/2}(S_{i,s}^*\otimes I_\cK)g
\end{equation*}
for every $g\in \ell^2(\FF_{n_1}^+\times\cdots \times \FF_{n_k}^+)\otimes \cK$, $i\in \{1,\ldots, k\}$ and $s\in \{1,\ldots, n_i\}$.
The operator  $C_{i,s}$ is well-defined on the range of $Y^{1/2}$ and
can be extended  by continuity   to the space $\cG$. Setting $M_{i,s}:=C_{i,s}^*$, we note that
$M=(M_1,\ldots, M_k)$, where $M_i=(M_{i,1},\ldots, M_{i,n_i})$,  is a pure element   in the regular $\Lambda$-polyball ${\bf B}_\Lambda(\cG)$.
Consequently, the associated noncommutative Berezin kernel $K_M:\cG\to \ell^2(\FF_{n_1}^+\times\cdots \times \FF_{n_k}^+)\otimes \overline{\Delta_M(I)\cG}$
is an isometry and
$$
K_M M_{i,s}^*=\left(S_{i,s}^*\otimes I_\cG\right) K_M
$$
for every $i \in \{1,\ldots, k\}$ and  $s\in \{1,\ldots, n_i\}$.
One can see that
the map
$$A:=Y^{1/2}K_M^*:\ell^2(\FF_{n_1}^+\times\cdots \times \FF_{n_k}^+)\otimes \overline{\Delta_M(I)\cG}\to\ell^2(\FF_{n_1}^+\times\cdots \times \FF_{n_k}^+)\otimes \cK
$$
is a multi-analytic operator and $Y=AA^*$.

Following the classical result of Beurling \cite{Be}, we say that $\cM\subset \ell^2(\FF_{n_1}^+\times\cdots \times \FF_{n_k}^+)\otimes \cK$   is a {\it Beurling type jointly invariant subspace} under the operators $S_{i,s}\otimes I_\cK$, where $i\in\{1,\ldots, k\}$ and  $s\in \{1,\ldots, n_i\}$,  if there is an inner multi-analytic operator
$\Psi: \ell^2(\FF_{n_1}^+\times\cdots \times \FF_{n_k}^+)\otimes \cL\to
\ell^2(\FF_{n_1}^+\times\cdots \times \FF_{n_k}^+)\otimes \cK$ such that
$$
\cM=\Psi \left(\ell^2(\FF_{n_1}^+\times\cdots \times \FF_{n_k}^+)\otimes \cL\right).
$$
In what follows, we use the notation $\left(({\bf S}_1\otimes I_\cK)|_\cM,\ldots, ({\bf S}_k\otimes I_\cK) |_\cM\right)$, where $$({\bf S}_i\otimes I_\cK)|_\cM:=((S_{i,1}\otimes I_\cK)|_\cM,\ldots, (S_{i,n_i}\otimes I_\cK)|_\cM), \qquad i\in \{1,\ldots, k\}.
$$

We proved in \cite{Po-twisted} the following characterization of the Beurling type jointly invariant subspaces under the universal model of the regular $\Lambda$-polyball.

\begin{theorem} \label{Beurling}  Let $\cM\subset \ell^2(\FF_{n_1}^+\times\cdots \times \FF_{n_k}^+)\otimes \cK$   be a jointly invariant subspace under $S_{i,s}\otimes I_\cK$, where $i\in\{1,\ldots, k\}$ and  $s\in \{1,\ldots, n_i\}$. Then the following statements are equivalent.
\begin{enumerate}
\item[(i)] $\cM$ is a Beurling type  jointly invariant subspace.

\item[(ii)]  $(id-\Phi_{{\bf S}_1\otimes I_\cK})\circ\cdots \circ (id-\Phi_{{\bf S}_k\otimes I_\cK})({\bf P}_\cM)\geq 0$, where ${\bf P}_\cM$ is the orthogonal projection onto $\cM$.

\item[(iii)] The  $k$-tuple $\left(({\bf S}_1\otimes I_\cK)|_\cM,\ldots, ({\bf S}_k\otimes I_\cK) |_\cM\right)$ is doubly $\Lambda$-commuting.
\item[(iv)] There is an isometric multi-analytic operator
$\Psi: \ell^2(\FF_{n_1}^+\times\cdots \times \FF_{n_k}^+)\otimes \cL\to
\ell^2(\FF_{n_1}^+\times\cdots \times \FF_{n_k}^+)\otimes \cK$ such that
$$
\cM=\Psi \left(\ell^2(\FF_{n_1}^+\times\cdots \times \FF_{n_k}^+)\otimes \cL\right).
$$
\end{enumerate}
\end{theorem}

  We say that  two
multi-analytic operators $A:\ell^2(\FF_{n_1}^+\times\cdots \times \FF_{n_k}^+)\otimes
 \cK_1 \to \ell^2(\FF_{n_1}^+\times\cdots \times \FF_{n_k}^+)\otimes \cK_2$ and
 $A':\ell^2(\FF_{n_1}^+\times\cdots \times \FF_{n_k}^+)\otimes \cK_1' \to
  \ell^2(\FF_{n_1}^+\times\cdots \times \FF_{n_k}^+)\otimes \cK_2'$  coincide if there are two unitary operators $u_j\in
B(\cK_j, \cK_j')$, $j=1,2$,  such that
$$
A'(I_{\ell^2(\FF_{n_1}^+\times\cdots \times \FF_{n_k}^+)}\otimes u_1)
=(I_{\ell^2(\FF_{n_1}^+\times\cdots \times \FF_{n_k}^+)}\otimes u_2) A.
$$

\begin{lemma}\label{fifi*}   Let $A_s:\ell^2(\FF_{n_1}^+\times\cdots \times \FF_{n_k}^+)\otimes \cH_s\to
\ell^2(\FF_{n_1}^+\times\cdots \times \FF_{n_k}^+)\otimes \cK$, \ $s=1,2$,    be  multi-analytic operators   with respect to ${\bf S}:=({\bf S}_1,\ldots, {\bf S}_k)$ such that
$
A_1 A_1^*=A_2 A_2^*.
$
Then there is a unique partial isometry $V:\cH_1\to \cH_2$ such that
$$A_1=A_2(I_{\ell^2(\FF_{n_1}^+\times\cdots \times \FF_{n_k}^+)}\otimes V),
$$
where $I_{\ell^2(\FF_{n_1}^+\times\cdots \times \FF_{n_k}^+)}\otimes V $ is an inner multi-analytic operator  with initial space $\supp (A_1)$ and  final space $\supp (A_2)$.
In particular,  the  multi-analytic operators $A_1|_{\supp (A_1)}$ and $A_2|_{\supp (A_2)}$ coincide.
\end{lemma}

\begin{proof} Using the definition of the universal model ${\bf S}:=({\bf S}_1,\ldots, {\bf S}_k)$, one can  easily prove  that
$(id-\Phi_{{\bf S}_1})\circ\cdots \circ(id-\Phi_{{\bf S}_k})(I)
={\bf P}_\CC$, where ${\bf P}_\CC$ is the
 orthogonal projection from $\ell^2(\FF_{n_1}^+\times\cdots \times \FF_{n_k}^+)$ onto $\CC 1\subset \ell^2(\FF_{n_1}^+\times\cdots \times \FF_{n_k}^+)$. Since $A_1, A_2$ are  multi-analytic operators   with respect to ${\bf S}$ and $A_1A_1^*=A_2 A_2^*$, we deduce that
 \begin{equation*}\begin{split}
 \|({\bf P}_\CC\otimes I_{\cH_1})A_1^*f\|^2&=\left<A_1(id-\Phi_{{\bf S}_1\otimes I})\circ\cdots \circ(id-\Phi_{{\bf S}_k\otimes I})(I)A_1^*f, f\right>\\
 &=\left<(id-\Phi_{{\bf S}_1\otimes I})\circ\cdots \circ(id-\Phi_{{\bf S}_k\otimes I})(A_1A_1^*)f, f\right>\\
 &=\left<(id-\Phi_{{\bf S}_1\otimes I})\circ\cdots \circ(id-\Phi_{{\bf S}_k\otimes I})(A_2A_2^*)f, f\right>\\
 &=\left<A_2(id-\Phi_{{\bf S}_1\otimes I})\circ\cdots \circ(id-\Phi_{{\bf S}_k\otimes I})(I)A_2^*f, f\right>\\
 &=\|({\bf P}_\CC\otimes I_{\cH_2})A_2^*f\|^2
 \end{split}
 \end{equation*}
 for all $ f\in \ell^2(\FF_{n_1}^+\times\cdots \times \FF_{n_k}^+)\otimes \cK$.
 Define $\cL_s:=({\bf P}_\CC\otimes I_{\cH_s})\overline{A_s^*(\ell^2(\FF_{n_1}^+\times\cdots \times \FF_{n_k}^+)\otimes \cK)}$, $s=1,2$, and  consider  the unitary  operator
 $U:\cL_1\to \cL_2$  defined by
 $$
 U({\bf P}_\CC\otimes I_{\cH_1})A_1^*f:=({\bf P}_\CC\otimes I_{\cH_2})A_2^*f, \qquad f\in \ell^2(\FF_{n_1}^+\times\cdots \times \FF_{n_k}^+)\otimes \cK.
 $$
 Now, we can extend $U$ to a partial isometry $V:\cH_1\to \cH_2$ with initial space $\cL_1={\supp (A_1)}$ and final space $\cL_2={\supp (A_2)}$. Moreover, we have
 $A_1 V^*= A_2|_{\CC\otimes \cH_2}$.   Since $A_1, A_2$ are  multi-analytic operators   with respect to ${\bf S}$, we deduce that
 $A_1(I_{\ell^2(\FF_{n_1}^+\times\cdots \times \FF_{n_k}^+)}\otimes V^*)=A_2$. The last part of the  lemma is obvious.
\end{proof}

We say that  ${T}=({ T}_1,\ldots, { T}_k)\in {\bf B}_\Lambda(\cH)$  has
  characteristic function   if there is a Hilbert space $\cE$ and
   a multi-analytic operator $\Psi:\ell^2(\FF_{n_1}^+\times\cdots \times \FF_{n_k}^+)\otimes \cE \to \ell^2(\FF_{n_1}^+\times\cdots \times \FF_{n_k}^+)\otimes \overline{{\Delta_{T}}(I) (\cH)}$ with respect to $S_{i,s}$, $i\in \{1,\ldots,k\}$, $s\in \{1,\ldots, n_i\}$, such that
$$
{ K_{T}}{ K_{T}^*} +\Psi \Psi^*=I,
$$
where 
$
K_{T}:\cH\to \ell^2(\FF_{n_1}^+\times\cdots \times \FF_{n_k}^+)\otimes \cD(T)
$
is the noncommutative Berezin kernel associated with ${T}$.
 According to Lemma \ref{fifi*}, if there is a characteristic function
  for ${T}\in {\bf B}_\Lambda(\cH)$, then it is essentially unique.

 \begin{theorem} \label{charact} A $k$-tuple  ${T}=({ T}_1,\ldots, { T}_k)$  in  the noncommutative $\Lambda$-polyball ${\bf B}_\Lambda(\cH)$  admits a characteristic function if and only if
$$
{\Delta}_{{\bf S}\otimes I}(I -{ K_{T}}{ K_{T}^*})\geq 0,
$$
  where ${K_{T}}$ is the noncommutative Berezin kernel  associated with ${T}$ and
  $${\Delta}_{{\bf S}\otimes I}:=(id-\Phi_{{\bf S}_1\otimes I})\circ\cdots \circ (id-\Phi_{{\bf S}_k\otimes I}).
  $$
   If, in addition,  ${ T}$ is  a pure $k$-tuple in ${\bf B}_\Lambda(\cH)$, then the following statements are equivalent.
   \begin{enumerate}
   \item[(i)]
     $T$ admits a characteristic function.
      \item[(ii)]  $({ K_{T}}\cH)^\perp$ is a Beurling type invariant subspace under all the operators ${ S}_{i,s}\otimes I$.
      \item[(iii)]  The $k$-tuple $({\bf S}\otimes I)|_{({ K_{T}}\cH)^\perp}$ is doubly $\Lambda$-commuting.
      \item[(iv)] There is a Beurling type invariant subspace $\cM$  under ${ S}_{i,s}\otimes I_\cD$  for some Hilbert space $\cD$ such that $T_{i,s}^*=(S_{i,s}^*\otimes I_\cD)|_{\cM^\perp}$  for any $i\in \{1,\ldots, k\}, s\in \{1,\ldots, n_i\}$ and
      $$
      \ell^2(\FF_{n_1}^+\times\cdots \times \FF_{n_k}^+)\otimes \cD=\bigvee_{{\boldsymbol \alpha}\in \FF_{n_1}^+\times\cdots \times \FF_{n_k}^+} ({\bf S}_{\boldsymbol \alpha}\otimes I_\cD)\cM^\perp.
      $$
    \end{enumerate}
\end{theorem}

\begin{proof}
 Assume that ${T}$
 has characteristic function. Then there is
  a multi-analytic operator $\Psi$  such that
$
{ K_{T}}{K_{T}^*} +\Psi \Psi^*=I.
$
Since $\Psi$ is  a multi-analytic operator and, ${ \Delta}_{{\bf S}\otimes I} (I)={\bf P}_\CC\otimes I$,
 we have
$$
{\Delta}_{{\bf S}\otimes I} (I -{ K_{T}}{ K_{T}^*})={ \Delta}_{{\bf S}\otimes I}(\Psi\Psi^*)
=\Psi{ \Delta}_{{\bf S}\otimes I} (I)\Psi^*=\Psi ({\bf P}_\CC\otimes I) \Psi^*\geq 0.
$$
 In order to prove the converse, we apply Theorem \ref{Beurling-fac} to the operator
  $Y=I -{ K_{T}}{ K_{T}^*}$.

  To prove the second part of the theorem, note that if $T$ is a pure  $k$-tuple in ${\bf B}_\Lambda$, then the Berezin kernel $K_T$ is an isometry and $I-K_TK_T^*={\bf P}_\cM$, where ${\bf P}_\cM$ is the orthogonal projection onto  $\cM:=({K_{T}}\cH)^\perp$. Using the first part of the theorem and
  applying
      Theorem \ref{Beurling}, one obtains the equivalences of the items (i), (ii), and (iii).

  Due to Theorem 5.6 from \cite{Po-twisted}, if
   $T=(T_1,\ldots, T_k)$ is  a pure $k$-tuple in the regular $\Lambda$-polyball and
  $$
K_{T}:\cH\to \ell^2(\FF_{n_1}^+\times\cdots \times \FF_{n_k}^+)\otimes \overline{\Delta_T(I)(\cH)},
$$
is
the corresponding noncommutative Berezin kernel, then   the dilation provided by  Theorem \ref{Berezin}  satisfies the relation
$$
K_T T_{i,s}^*=\left(S_{i,s}^*\otimes I_{\cD(T)}\right) K_T, \quad i\in \{1,\ldots, k\}, s\in \{1,\ldots, n_i\},
$$
and it is minimal, i.e.
$$
\ell^2(\FF_{n_1}^+\times\cdots \times \FF_{n_k}^+)\otimes \overline{\Delta_T(I)(\cH)} =\bigvee_{\boldsymbol\alpha\in \FF_{n_1}^+\times \cdots \times \FF_{n_k}^+} ({\bf S}_{\boldsymbol\alpha}\otimes
I_{\cD(T)}) K_T\cH.
$$
Moreover, this dilation   is unique up to an isomorphism.
Setting $\cM:= (K_{T}\cH)^\perp$, $\cD:=\cD(T):=\overline{\Delta_T(I)(\cH)}$, and identifying $\cH$ with   $K_T\cH$, we conclude that  (ii)$\implies$ (iv).
  Now, we prove the implication (iv)$\implies$(ii).

  Assume that  ${ T}\in {\bf B}_\Lambda(\cH)$  is a pure element and  that there is  a Beurling type invariant subspace $\cM$ under ${ S}_{i,s}\otimes I_\cD$ such that $T_{i,s}^*=({S}_{i,s}^*\otimes I_\cD)|_{\cM^\perp}$ and
$
\ell^2(\FF_{n_1}^+\times\cdots \times \FF_{n_k}^+)\otimes \cD=\bigvee_{\boldsymbol\alpha\in \FF_{n_1}^+\times \cdots \times \FF_{n_k}^+} ({\bf S}_{\boldsymbol\alpha}\otimes
I_{\cD}) \cM^\perp.
$
Using the uniqueness of the dilation provided by the noncommutative Berezin kernel associate
with $T$, we deduce that  there is a unitary  operator $\Omega:\cD(T)\to \cD$ such that
$(I_{\ell^2(\FF_{n_1}^+\times\cdots \times \FF_{n_k}^+)}\otimes \Omega)K_T\cH=\cM^\perp$.
Hence,
$(I_{\ell^2(\FF_{n_1}^+\times\cdots \times \FF_{n_k}^+)}\otimes \Omega)K_TK_T^*
(I_{\ell^2(\FF_{n_1}^+\times\cdots \times \FF_{n_k}^+)}\otimes \Omega^*)={\bf P}_{\cM^\perp}$. Since
$\cM$ is a Beuling type invariant subspace,
there is an inner multi-analytic operator
$\Psi: \ell^2(\FF_{n_1}^+\times\cdots \times \FF_{n_k}^+)\otimes \cL\to
\ell^2(\FF_{n_1}^+\times\cdots \times \FF_{n_k}^+)\otimes \cD$ such that
$$
{\bf P}_\cM=\Psi  \Psi^*.
$$
Now, one can easily see that
$$
I-K_TK_T^*=\Phi\Phi^*,
$$
where $\Phi:=(I_{\ell^2(\FF_{n_1}^+\times\cdots \times \FF_{n_k}^+)}\otimes \Omega^*)\Psi$  is an inner multi-analytic operator.
  The proof is complete.
\end{proof}

 If ${T}$ has characteristic function, the multi-analytic operator $A$ provided by  Theorem \ref{Beurling-fac} when $Y=I -{K_{T}}{ K_{T}^*}$, which we denote by $\Theta_{T}$,  is called  the {\it characteristic function} of ${T}$. More precisely,  due to the remarks following Theorem \ref{Beurling-fac}, one can see that $\Theta_{T}$
  is the multi-analytic operator
 $$\Theta_{T}:\ell^2(\FF_{n_1}^+\times\cdots \times \FF_{n_k}^+)\otimes \overline{{\Delta_{M_T} }(I)(\cM_T)} \to \ell^2(\FF_{n_1}^+\times\cdots \times \FF_{n_k}^+)\otimes \overline{{\Delta_{T} }(I)(\cH)}
 $$
defined by $\Theta_{T}:=(I -{ K_{T}}{K_{T}^*})^{1/2} { K_{M_T}^*}$, where
$${K_{T}}: \cH \to \ell^2(\FF_{n_1}^+\times\cdots \times \FF_{n_k}^+) \otimes  \overline{{\Delta_{T}}(I)(\cH)}$$
is the noncommutative Berezin kernel associated with $T\in {\bf B}_\Lambda(\cH)$ and
$${ K_{M_T}}: \cH \to \ell^2(\FF_{n_1}^+\times\cdots \times \FF_{n_k}^+) \otimes  \overline{{\Delta_{M_T}}(I)(\cM_{T})}$$
is the noncommutative Berezin kernel associated
 with ${M_T}\in {\bf B}_{\Lambda}(\cM_{T})$.  Here, we have
$$
\cM_{T}:= \overline{{\rm range}\,(I -{K_{T}}{ K_{T}^*}) }
$$
and  ${ M_T}:=(M_1,\ldots, M_k)$ is the $k$-tuple
 with $M_i:=(M_{i,1},\ldots, M_{i,n_i})$ and $M_{i,s}\in B(\cM_{T})$   given by $M_{i,s}:=A_{i,s}^*$, where $A_{i,s}\in B(\cM_{T})$ is uniquely defined by
$$
A_{i,s}\left[(I -{K_{T}}{K_{T}^*})^{1/2}f\right]:=(I -{K_{T}}{K_{T}^*})^{1/2}({S}_{i,s}\otimes I)f
$$
for all $f\in \ell^2(\FF_{n_1}^+\times\cdots \times \FF_{n_k}^+)\otimes \overline{{\Delta_{T}}(I)(\cH)}$. According to Theorem \ref{Beurling-fac}, we have
$
{K_{T}}{K_{T}^*}+ \Theta_{T}\Theta_{T}^*=I.
$

\begin{theorem}\label{pure-model} Let  ${ T}=({ T}_1,\ldots, { T}_k)$ be a  $k$-tuple  in
${\bf B}_\Lambda(\cH)$ which  admits characteristic function. Then ${T}$  is pure if and only if  the characteristic function $\Theta_{T}$ is an inner multi-analytic operator. Moreover, in this case  ${T}=({ T}_1,\ldots, { T}_k)$ is unitarily equivalent to ${G}=({ G}_1,\ldots, { G}_k)$, where $G_i:=(G_{i,1},\ldots, G_{i,n_i})$ is defined by
$$ G_{i,s}:={\bf P}_{H_{T}} \left({ S}_{i,s}\otimes I\right)|_{H_{T}}, \quad i\in \{1,\ldots, k\}, s\in \{1,\ldots, n_i\},
$$
and ${\bf P}_{ H_{T}}$ is the orthogonal projection of $ \ell^2(\FF_{n_1}^+\times\cdots \times \FF_{n_k}^+)\otimes \overline{{ \Delta_{T}}(I)(\cH)}$ onto
$${ H_{T}}:=\left\{ \ell^2(\FF_{n_1}^+\times\cdots \times \FF_{n_k}^+)\otimes \overline{{\Delta_{T}}(I)(\cH)}\right\}\ominus {\rm range}\, \Theta_{T}.
$$
\end{theorem}
\begin{proof} Assume that ${T}$ is a pure $k$-tuple  in
${\bf B}_\Lambda(\cH)$ which  admits characteristic function.    Theorem \ref{Berezin} shows that
 the noncommutative Berezin kernel associated with ${T}$, i.e.,
$${K_{T}}: \cH \to   \ell^2(\FF_{n_1}^+\times\cdots \times \FF_{n_k}^+) \otimes  \overline{{\Delta_{T}}(I)(\cH)}$$
is an isometry, the subspace ${ K_{T}}\cH$ is coinvariant  under  the operators
${S}_{i,s}\otimes I_{\overline{{\Delta_{T}}(I)(\cH)}}$, $i\in \{1,\ldots, k\}$, $s\in \{1,\ldots, n_i\}$, and
$T_{i,s}={ K_{T}^*}({S}_{i,s}\otimes I_{\overline{{\Delta_{T}}(I)(\cH)}}) {K_{T}}$.
Since ${ K_{T}}{ K_{T}^*}$ is the orthogonal projection of $ \ell^2(\FF_{n_1}^+\times\cdots \times \FF_{n_k}^+)\otimes \overline{{ \Delta_{T}}(I)(\cH)}$ onto ${ K_{T}}\cH$ and  ${K_{T}}{ K_{T}^*}+\Theta_{T}\Theta_{T}^*=I $, we deduce that $\Theta_{T}$ is a partial isometry and
${ K_{T}}\cH={ H_{T}}$. Since ${K_{T}}$ is an isometry, we can identify $\cH$ with ${ K_{T}}\cH$. Therefore, ${ T}=({ T}_1,\ldots, { T}_k)$ is unitarily equivalent to ${ G}=({ G}_1,\ldots, { G}_k)$.

Conversely,   assume that  $\Theta_{T}$ is an inner multi-analytic operator. Since  $
{ K_{T}}{K_{T}^*}+ \Theta_{T}\Theta_{T}^*=I,
$ and $\Theta_T$ is a partial isometry,
 the noncommutative Berezin kernel ${K_{T}}$ is a partial isometry. On the other hand,
 since ${T}$ is completely non-coisometric,
  ${K_{T}}$ is a one-to-one partial isometry and,
consequently, an isometry. Due to  Theorem  \ref{Berezin}, we have
$$
{K_{T}^*}{ K_{T}}=
\lim_{{\bf q}=(q_1,\ldots, q_k)\in \ZZ_+^k}(id-\Phi_{T_k}^{q_k})\cdots (id-\Phi_{T_1}^{q_1})(I)=I.
$$
Hence, we deduce that ${ T}$ is a pure $k$-tuple in ${\bf B}_\Lambda(\cH)$. The proof is complete.
\end{proof}

Now, we are able to provide a model theorem for the class of    completely non-coisometric $k$-tuple  of operators  in
${\bf B}_\Lambda(\cH)$ which  admit characteristic functions.

\begin{theorem}\label{model}  Let ${T}=({ T}_1,\ldots, { T}_k)$ be a
 completely non-coisometric $k$-tuple  in the $\Lambda$-polyball
${\bf B}_\Lambda(\cH)$  which  admits characteristic function, and let  ${\bf S}:=({\bf S}_1,\ldots, {\bf S}_k)$ be the universal model associated to ${\bf B}_\Lambda(\cH)$.   Set
$$
\cD:=\overline{{ \Delta_{T}}(I)(\cH)},\quad  \quad \cD_*:=\overline{{\Delta_{M_T}}(I)(\cM_T)},
$$
and $D_{\Theta_{T}}:= \left(I-\Theta_{T}^*
\Theta_{T}\right)^{1/2}$, where $\Theta_{T}$ is the characteristic function of ${T}$.
 Then ${T} $ is unitarily equivalent to
$\GG:=(\GG_1,\ldots, \GG_k)\in {\bf B}_\Lambda(\HH_{T})$, where $\GG_i:=(\GG_{i,1},\ldots, \GG_{i,n_i})$ and, for each $i\in \{1,\ldots, k\}$ and $s\in \{1,\ldots, n_i\}$,  $\GG_{i,s}$ is a bounded operator acting on the
Hilbert space
\begin{equation*}
\begin{split}
\HH_{T}&:=\left[\left(\ell^2(\FF_{n_1}^+\times\cdots \times \FF_{n_k}^+)\otimes\cD\right)\bigoplus
\overline{D_{\Theta_{T}}(\ell^2(\FF_{n_1}^+\times\cdots \times \FF_{n_k}^+)\otimes \cD_*)}\right]\\
& \qquad \qquad\ominus\left\{\Theta_{T}\varphi\oplus
D_{\Theta_{T}}\varphi:\ \varphi\in \ell^2(\FF_{n_1}^+\times\cdots \times \FF_{n_k}^+)\otimes  \cD_*\right\}
\end{split}
\end{equation*}
 and is uniquely defined by the relation
$$
\left( {\bf P}^{\KK_T}_{\ell^2(\FF_{n_1}^+\times\cdots \times \FF_{n_k}^+)\otimes\cD}|_{\HH_{T}}\right) \GG_{i,s}^*f=
({S}_{i,s}^*\otimes I_{ \cD})\left( {\bf P}^{\KK_T}_{\ell^2(\FF_{n_1}^+\times\cdots \times \FF_{n_k}^+)\otimes \cD}|_{\HH_{T}}\right)f
$$
for every $f\in \HH_{T}$.  Here,
   $ {\bf P}^{\KK_T}_{\ell^2(\FF_{n_1}^+\times\cdots \times \FF_{n_k}^+)\otimes  \cD}$ is the orthogonal
projection of the Hilbert space
$$\KK_{T}:=\left(\ell^2(\FF_{n_1}^+\times\cdots \times \FF_{n_k}^+)\otimes\cD\right)\bigoplus
\overline{D_{\Theta_{T}}(\ell^2(\FF_{n_1}^+\times\cdots \times \FF_{n_k}^+)\otimes \cD_*)}$$
 onto
the subspace $\ell^2(\FF_{n_1}^+\times\cdots \times \FF_{n_k}^+)\otimes \cD$ and ${\bf P}^{\KK_T}_{\ell^2(\FF_{n_1}^+\times\cdots \times \FF_{n_k}^+)\otimes\cD}|_{\HH_{T}}$ is a one-to-one operator.
 \end{theorem}

\begin{proof}
 A straightforward computation reveals  that the operator $\Psi: \ell^2(\FF_{n_1}^+\times\cdots \times \FF_{n_k}^+)\otimes \cD\to
\cK_{T}$ defined by
$$
\Psi \varphi:=\Theta_{T} \varphi\oplus D_{\Theta_{T}}
\varphi,\quad \varphi\in \ell^2(\FF_{n_1}^+\times\cdots \times \FF_{n_k}^+)\otimes \cD_*,
$$
 is an isometry and
\begin{equation}\label{fi}
\Psi^*(g\oplus 0)=\Theta_{T}^*g, \qquad g\in \ell^2(\FF_{n_1}^+\times\cdots \times \FF_{n_k}^+)\otimes
\cD.
\end{equation}
Consequently, we deduce that
\begin{equation*}
\begin{split}
\|g\|^2&= \|{\bf P}^{\KK_T}_{\HH_{T}}(g\oplus 0)\|^2+\|\Psi \Psi^*(g\oplus 0)\|^2
=\|{\bf P}^{\KK_T}_{\HH_{T}}(g\oplus 0)\|^2+\|\Theta_{T}^*g\|^2
\end{split}
\end{equation*}
for every $g\in \ell^2(\FF_{n_1}^+\times\cdots \times \FF_{n_k}^+)\otimes \cD$, where ${\bf P}^{\KK_T}_{\HH_{T}}$ the orthogonal projection of
$\KK_{T}$ onto the subspace $\HH_{T}$. Since
\begin{equation*}
 \|{ K_{T}^*}  g\|^2+ \|\Theta_{T}^*g\|^2=\|g\|^2,
\quad g\in\ell^2(\FF_{n_1}^+\times\cdots \times \FF_{n_k}^+)\otimes \cD,
\end{equation*}
 we have
\begin{equation}\label{K*P}
\|{ K_{T}^*} g\|=\|{\bf P}^{\KK_T}_{\HH_{T}}(g\oplus 0)\|,  \quad g\in
\ell^2(\FF_{n_1}^+\times\cdots \times \FF_{n_k}^+) \otimes \cD.
\end{equation}
Due to the fact that  the $k$-tuple  ${T}=(T_1,\ldots, T_k)$ is  completely non-coisometric in ${\bf B}_\Lambda(\cH)$,  the noncommutative Berezin kernel ${K_{T}}$ is  a one-to-one
operator.  Thus  $\overline{\text{\rm
range}\, {K_{T}^*}}=\cH$. Let
$f\in \HH_{T}$ be with the property  that $\left<f, {\bf P}^{\KK_T}_{\HH_{T}}(g\oplus
0)\right>=0$ for any $g\in \ell^2(\FF_{n_1}^+\times\cdots \times \FF_{n_k}^+)\otimes \cD $. Due to the definition of
$\HH_{T}$ and the fact that $\KK_{T}$ coincides with the
closed span of all vectors
$g\oplus 0$, for $ g\in \ell^2(\FF_{n_1}^+\times\cdots \times \FF_{n_k}^+)\otimes \cD$, and
$\Theta_{T} \varphi\oplus D_{\Theta_{T}}
\varphi$, for  $ \varphi\in \ell^2(\FF_{n_1}^+\times\cdots \times \FF_{n_k}^+)\otimes \cD$,
we must have $f=0$. Consequently,
$$
\HH_{T}=\left\{{\bf P}^{\KK_T}_{\HH_{T}}(g\oplus 0):\ g\in \ell^2(\FF_{n_1}^+\times\cdots \times \FF_{n_k}^+)\otimes
\cD\right\}.
$$
Now, using  relation \eqref{K*P},  we deduce  that
 there is a  unitary
operator $\Gamma:\cH\to \HH_{T}$ such that
\begin{equation}\label{Ga}
\Gamma({ K_{T}^*} g)={\bf P}^{\KK_T}_{\HH_{T}}(g\oplus 0), \qquad  g\in
\ell^2(\FF_{n_1}^+\times\cdots \times \FF_{n_k}^+)\otimes \cD.
\end{equation}
For each $i\in \{1,\ldots, k\}$ and $s\in\{1,\ldots, n_i\}$, we define the operator $\GG_{i,s}:\HH_{T}\to \HH_{T}$ by relation
$$\GG_{i,s}:=\Gamma
T_{i,s}\Gamma^*,\qquad  i\in \{1,\ldots,k \}, s\in \{1,\ldots,n_i\}.
$$
Since $T\in {\bf B}_\Lambda(\cH)$, we also have $\GG\in {\bf B}_\Lambda(\cH)$.
 The next step is to show that
\begin{equation}\label{PN-intert} 
 \left( {\bf P}^{\KK_T}_{\ell^2(\FF_{n_1}^+\times\cdots \times \FF_{n_k}^+)\otimes \cD}|_{\HH_{T}}\right) \GG_{i,s}^*f=
({ S}_{i,s}^*\otimes I_\cD)\left( {\bf P}^{\KK_T}_{\ell^2(\FF_{n_1}^+\times\cdots \times \FF_{n_k}^+)\otimes
\cD}|_{\HH_{T}}\right)f
\end{equation}
for every   $i\in \{1,\ldots, k\}$, $s\in\{1,\ldots, n_i\}$, and $f\in \HH_{T}$.
Taking into account   relations \eqref{Ga} and \eqref{fi},
and   that $\Psi$ is an isometry, we obtain
\begin{equation*}
\begin{split}
{\bf P}^{\KK_T}_{\ell^2(\FF_{n_1}^+\times\cdots \times \FF_{n_k}^+)\otimes \cD} \Gamma {K_{T}^*} g&= {\bf P}^{\KK_T}_{\ell^2(\FF_{n_1}^+\times\cdots \times \FF_{n_k}^+)\otimes
\cD} {\bf P}^{\KK_T}_{\HH_{T}}(g\oplus 0)\\
&
 =
g-{\bf P}^{\KK_T}_{\ell^2(\FF_{n_1}^+\times\cdots \times \FF_{n_k}^+)\otimes \cD} \Psi \Psi^*(g\oplus 0)\\
&=g-\Theta_{T} \Theta_{T}^* g ={ K_{T}} { K_{T}^*}g
\end{split}
\end{equation*}
for all $g\in \ell^2(\FF_{n_1}^+\times\cdots \times \FF_{n_k}^+)\otimes \cD$. Since 
$\overline{\text{\rm range}\,{K_{T}^*}}=\cH$, we obtain
\begin{equation}\label{PGK} {\bf P}^{\KK_T}_{\ell^2(\FF_{n_1}^+\times\cdots \times \FF_{n_k}^+)\otimes \cD} \Gamma={ K_{T}}.
\end{equation}
On the other hand, since $T$ is a completely non-coisometric tuple, the  noncommutative Berezin kernel
${ K_{T}}$ is one-to-one. Now, relation \eqref{PGK} implies   
\begin{equation*}
 {\bf P}^{\KK_T}_{\ell^2(\FF_{n_1}^+\times\cdots \times \FF_{n_k}^+)\otimes \cD} |_{\HH_{T}}={K_{T}}
\Gamma^*
\end{equation*}
and  shows that   ${\bf P}^{\KK_T}_{\ell^2(\FF_{n_1}^+\times\cdots \times \FF_{n_k}^+)\otimes \cD} |_{\HH_{T}}$ is  a one-to-one operator acting from $\HH_{T}$ to $\ell^2(\FF_{n_1}^+\times\cdots \times \FF_{n_k}^+)\otimes
\cD$. Hence, using relation \eqref{PGK} and  Theorem \ref{Berezin}, we deduce that
\begin{equation*}
\begin{split}
\left({\bf P}^{\KK_T}_{\ell^2(\FF_{n_1}^+\times\cdots \times \FF_{n_k}^+)\otimes \cD} |_{\HH_{T}}\right)
\GG_{i,s}^*\Gamma h&= \left({\bf P}^{\KK_T}_{\ell^2(\FF_{n_1}^+\times\cdots \times \FF_{n_k}^+)\otimes \cD}
|_{\HH_{T}}\right) \Gamma T_{i,s}^* h \\
&={ K_{T}} T_{i,s}^*h\\&=
\left( { S}_{i,s}^*\otimes I_{\cD}\right) { K_{T}}h\\
&= \left( {S}_{i,s}^*\otimes I_{\cD}\right)
\left({\bf P}^{\KK_T}_{\ell^2(\FF_{n_1}^+\times\cdots \times \FF_{n_k}^+)\otimes \cD} |_{\HH_{T}}\right)\Gamma h
\end{split}
\end{equation*}
for every $i\in \{1,\ldots, k\}$, $s\in\{1,\ldots, n_i\}$, and $h\in \cH$. Therefore,    relation \eqref{PN-intert} holds.
We remark  that, since the operator ${\bf P}^{\KK_T}_{\ell^2(\FF_{n_1}^+\times\cdots \times \FF_{n_k}^+)\otimes
\cD}|_{\HH_{T}}$ is one-to-one, the
relation \eqref{PN-intert} uniquely determines each operator
$\GG_{i,s}^*$ for all $i\in \{1,\ldots, k\}$ and $s\in \{1,\ldots, n_i\}$.
The  proof is complete.
\end{proof}

Now, we show  that the  characteristic function
$\Theta_{T}$  is a complete unitary invariant for the completely non-coisometric $k$-tuples
in ${\bf B}_\Lambda(\cH)$ which admit characteristic functions.
\begin{theorem}\label{u-inv}
Let   ${T}:=(T_1,\ldots,
T_k)\in  {\bf B}_\Lambda(\cH)$ and ${T}':=(T_1',\ldots, T_k')\in
 {\bf B}_\Lambda(\cH')$ be two completely non-coisometric $k$-tuples which admit characteristic functions. Then ${T}$ and
${ T}'$ are unitarily equivalent if and only if their
characteristic functions $\Theta_{T}$ and $\Theta_{T'}$
coincide.
\end{theorem}

\begin{proof} To prove the direct implication of the theorem, 
assume that  the $k$-tuples ${T}$ and ${T}'$ are unitarily
equivalent.  Let $W:\cH\to \cH'$ be a unitary operator such that
$T_{i,s}=W^*T_{i,s}'W$ for every $i\in\{1,\ldots,k\}$ and $s\in\{1,\ldots, n_i\}$.
Note   that $W{\Delta_{T}}(I)={ \Delta_{T'}}(I) W$ and $W\cD=\cD'$, where the subspaces $\cD$ and $\cD'$ are given by 
$$
\cD:=\overline{{ \Delta_{T}}(I)(\cH)},\quad
 \quad \cD':=\overline{{\Delta_{T'}}(I)(\cH')}.
$$
 On the other hand, using the definition of the noncommutative Berezin kernel
  associated with $\Lambda$-polyballs, it is easy to see  that $(I_{\ell^2(\FF_{n_1}^+\times\cdots \times \FF_{n_k}^+)}\otimes W)K_{T}= K_{T'}W$.  Consequently,
\begin{equation*}
 (I_{\ell^2(\FF_{n_1}^+\times\cdots \times \FF_{n_k}^+)}\otimes W)(I- {K_{T}}{ K_{T}^*}) (I_{\ell^2(\FF_{n_1}^+\times\cdots \times \FF_{n_k}^+)}\otimes W)=I- { K_{T'}}{K_{T'}^*}
\end{equation*}
 and $(I_{\ell^2(\FF_{n_1}^+\times\cdots \times \FF_{n_k}^+)}\otimes W)\cM_{T}= \cM_{T'}$,
 where
 $
\cM_{T}:= \overline{{\rm range}\,(I -{K_{T}}{K_{T}^*}) }
$
and $\cM_{T'}:=\overline{{\rm range}\,(I -{K_{T'}}{K_{T'}^*}) }$. As mentioned in the remarks preceding   Theorem \ref{pure-model},    ${M_T}:=(M_1,\ldots, M_k)\in {\bf B}_\Lambda(\cM_T)$ is
 the $k$-tuple with $M_i:=(M_{i,1},\ldots, M_{i,n_i})$ and
  $M_{i,s}\in B(\cM_{T})$, where  $M_{i,s}:=A_{i,s}^*$ and  $A_{i,s}\in B(\cM_{ T})$ is uniquely defined by relation 
$$
A_{i,s}\left[(I -{ K_{T}}{K_{T}^*})^{1/2}x\right]:=(I -{ K_{T}}{ K_{T}^*})^{1/2}({ S}_{i,s}\otimes I)x
$$
for all $x\in \ell^2(\FF_{n_1}^+\times\cdots \times \FF_{n_k}^+)\otimes \overline{{ \Delta_{T}}(I)(\cH)}$. In a similar manner, we define the $k$-tuple ${ M_{T'}}\in {\bf B}_\Lambda(\cM_{T'})$ and the operators $A_{i,s}'\in B(\cM_{T'})$.
It is easy to see that
\begin{equation*}
\begin{split}
A_{i,s} (I -{K_{T}}{ K_{T}^*})^{1/2}f
&=(I_{\ell^2(\FF_{n_1}^+\times\cdots \times \FF_{n_k}^+)}\otimes W^*) A_{i,s}'(I -{ K_{T'}}{K_{T'}^*})^{1/2}(I_{\ell^2(\FF_{n_1}^+\times\cdots \times \FF_{n_k}^+)}\otimes W^*)f\\
&=(I_{\ell^2(\FF_{n_1}^+\times\cdots \times \FF_{n_k}^+)}\otimes W^*) A_{i,s}' (I_{\ell^2(\FF_{n_1}^+\times\cdots \times \FF_{n_k}^+)}\otimes W)(I -{K_{T}}{ K_{T}^*})^{1/2}f
\end{split}
\end{equation*}
 for all $f\in \ell^2(\FF_{n_1}^+\times\cdots \times \FF_{n_k}^+)\otimes \overline{{\Delta_{T}}(I)(\cH)}$.
This  implies
 $$
 A_{i,s}=(I_{\ell^2(\FF_{n_1}^+\times\cdots \times \FF_{n_k}^+)}\otimes W^*) A_{i,s}' (I_{\ell^2(\FF_{n_1}^+\times\cdots \times \FF_{n_k}^+)}\otimes W).
$$
 It is clear now that  $(I_{\ell^2(\FF_{n_1}^+\times\cdots \times \FF_{n_k}^+)}\otimes W)\cD_*=\cD_*'$, where
 $\cD_*:=\overline{{\Delta_{M_T}}(I)(\cM_T)}$ and $\cD_*':=\overline{{\Delta_{M_{T'}}}(I)(\cM_{T'})}$.
 Define the
unitary operators $u$ and $u'$ by setting
$$u:=W|_\cD:\cD\to \cD' \ \text{ and }\
u_*:=(I_{\ell^2(\FF_{n_1}^+\times\cdots \times \FF_{n_k}^+)}\otimes W)|_{\cD_*}:\cD_*\to {\cD_*'}.
$$
Straightforward calculations reveal that 
$$
(I_{\ell^2(\FF_{n_1}^+\times\cdots \times \FF_{n_k}^+)}\otimes
u)\Theta_{T}=\Theta_{T'}(I_{\ell^2(\FF_{n_1}^+\times\cdots \times \FF_{n_k}^+)}\otimes u_*),
$$
which proves the direct implication of the theorem.
Conversely, assume that the  characteristic functions  of
${T}$ and ${T'}$ coincide.  In this case,   there exist unitary operators
$u:\cD\to \cD'$ and $u_*:\cD_*\to
\cD_*'$ such that
\begin{equation*}
(I_{\ell^2(\FF_{n_1}^+\times\cdots \times \FF_{n_k}^+)}\otimes u)\Theta_{T}
=\Theta_{T'}(I_{\ell^2(\FF_{n_1}^+\times\cdots \times \FF_{n_k}^+)}\otimes
u_*).
\end{equation*}
Hence, we deduce that
$$
D_{\Theta_{T}}=\left(I_{\ell^2(\FF_{n_1}^+\times\cdots \times \FF_{n_k}^+)}\otimes u_*\right)^*
D_{\Theta_{T'}}\left(I_{\ell^2(\FF_{n_1}^+\times\cdots \times \FF_{n_k}^+)}\otimes u_*\right)
$$
and
$$
\left(I_{\ell^2(\FF_{n_1}^+\times\cdots \times \FF_{n_k}^+)}\otimes
u_*\right)\overline{D_{\Theta_{T}}(\ell^2(\FF_{n_1}^+\times\cdots \times \FF_{n_k}^+)\otimes
\cD_*)}= \overline{D_{\Theta_{T'}}(\ell^2(\FF_{n_1}^+\times\cdots \times \FF_{n_k}^+)\otimes
\cD_*')},
$$
where $D_{\Theta_{T}}:=(I-\Theta_T^*\Theta_T)^{1/2}$.
 Define now the unitary operator $U:\cK_{T}\to \cK_{T'}$
by setting
$$U:=(I_{\ell^2(\FF_{n_1}^+\times\cdots \times \FF_{n_k}^+)}\otimes u)\oplus (I_{\ell^2(\FF_{n_1}^+\times\cdots \times \FF_{n_k}^+)}\otimes u_*).
$$
It is easy to see  that the operator $\Psi:\ell^2(\FF_{n_1}^+\times\cdots \times \FF_{n_k}^+)\otimes
\cD_*\to \KK_{T}$, defined  by $$ \Psi\varphi:=
\Theta_{T}\varphi \oplus D_{\Theta_{T}}\varphi,\quad
\varphi\in \ell^2(\FF_{n_1}^+\times\cdots \times \FF_{n_k}^+)\otimes \cD_*, $$
 and the corresponding operator $\Psi'$ satisfy the   relations
\begin{equation}
\label{Uni1} U \Psi\left(I_{\ell^2(\FF_{n_1}^+\times\cdots \times \FF_{n_k}^+)}\otimes u_*\right)^*=\Psi'
\end{equation}
and
\begin{equation}
\label{Uni2} \left(I_{\ell^2(\FF_{n_1}^+\times\cdots \times \FF_{n_k}^+)}\otimes u\right)
 P_{\ell^2(\FF_{n_1}^+\times\cdots \times \FF_{n_k}^+)\otimes
\cD}^{\KK_{T}} U^*=P_{\ell^2(\FF_{n_1}^+\times\cdots \times \FF_{n_k}^+)\otimes
\cD'}^{\KK_{T'}},
\end{equation}
where $P_{\ell^2(\FF_{n_1}^+\times\cdots \times \FF_{n_k}^+)\otimes \cD}^{\KK_{T}}$ is the orthogonal
projection of $\KK_{T}$ onto $\ell^2(\FF_{n_1}^+\times\cdots \times \FF_{n_k}^+)\otimes \cD$. On the other hand  relation \eqref{Uni1} implies
\begin{equation*}
\begin{split}
U\HH_{T}&=U\KK_{T}\ominus U\Psi(\ell^2(\FF_{n_1}^+\times\cdots \times \FF_{n_k}^+)\otimes \cD_*)\\
&=\KK_{T'}\ominus \Psi'(I_{\ell^2(\FF_{n_1}^+\times\cdots \times \FF_{n_k}^+)}\otimes u_*)(\ell^2(\FF_{n_1}^+\times\cdots \times \FF_{n_k}^+)\otimes
 \cD_*)\\
&=\KK_{T'}\ominus \Psi' (\ell^2(\FF_{n_1}^+\times\cdots \times \FF_{n_k}^+)\otimes \cD_*').
\end{split}
\end{equation*}
Consequently,  $U|_{\HH_{T}}:\HH_{T}\to
\HH_{T'}$ is  a unitary operator.
We remark  that
\begin{equation}
\label{intertw} ({S}_{i,s}^*\otimes I_{\cD'})(I_{\ell^2(\FF_{n_1}^+\times\cdots \times \FF_{n_k}^+)}\otimes
u)= (I_{\ell^2(\FF_{n_1}^+\times\cdots \times \FF_{n_k}^+)}\otimes u)({S}_{i,s}^*\otimes I_{\cD})
\end{equation}
for every $i\in \{1,\ldots, k\}$ and  $s\in \{1,\ldots, n_i\}$.
Let $\GG:=(\GG_1,\ldots, \GG_n)$ and $\GG':=(\GG_1',\ldots, \GG_n')$
be the model operators provided by Theorem \ref{model}  for ${T}$ and
${T}'$, respectively. Taking into account  relations \eqref{Uni2}, \eqref{intertw}, and relation  \eqref{PN-intert}  for ${T}'$
and ${T}$,  respectively, we deduce that
\begin{equation*}
\begin{split}
P_{\ell^2(\FF_{n_1}^+\times\cdots \times \FF_{n_k}^+)\otimes \cD'}^{\KK_{T'}}{\GG_{i,s}'}^*Uf
&=  ({ S}_{i,s}^*\otimes I_{\cD'}) P_{\ell^2(\FF_{n_1}^+\times\cdots \times \FF_{n_k}^+)\otimes
\cD}^{\KK_{T}}Ux\\
&=({S}_{i,s}^*\otimes I_{\cD'})(I_{\ell^2(\FF_{n_1}^+\times\cdots \times \FF_{n_k}^+)}\otimes u)
P_{\ell^2(\FF_{n_1}^+\times\cdots \times \FF_{n_k}^+)\otimes \cD}^{\KK_{T}}f\\
&=(I_{\ell^2(\FF_{n_1}^+\times\cdots \times \FF_{n_k}^+)}\otimes u)({S}_{i,s}^*\otimes I_{\cD})
P_{\ell^2(\FF_{n_1}^+\times\cdots \times \FF_{n_k}^+)\otimes \cD}^{\KK_{T}}f\\
&=(I_{\ell^2(\FF_{n_1}^+\times\cdots \times \FF_{n_k}^+)}\otimes u) P_{\ell^2(\FF_{n_1}^+\times\cdots \times \FF_{n_k}^+)\otimes \cD}^{\KK_{T}}
\GG_i^*f\\
&= P_{\ell^2(\FF_{n_1}^+\times\cdots \times \FF_{n_k}^+)\otimes \cD'}^{\KK_{T'}}U \GG_{i,s}^*f
\end{split}
\end{equation*}
for every  $f\in \HH_T$,  $i=\{1,\ldots, k\}$,  and $s\in\{1,\ldots, n_i\}$. According to Theorem \ref{model},
$P_{\ell^2(\FF_{n_1}^+\times\cdots \times \FF_{n_k}^+)\otimes \cD'}^{\KK_{T'}}|_{\HH_{T'}}$ is a
one-to-one operator. Consequently, the relations above imply
$
\left(U|_{\HH_{T}}\right)
\GG_{i,s}^*=({\GG_{i,s}'})^*\left(U|_{\HH_{T}}\right)
$
for every $i\in \{1,\ldots, k\}$ and $s\in \{1,\ldots, n_i\}$.
 Using   Theorem \ref{model}, we conclude that the $k$-tuples
  ${T}$ and ${T}'$ are
 unitarily equivalent. This   completes the proof.
\end{proof}

\begin{corollary}
If  ${T}:=(T_1,\ldots,
T_k)\in  {\bf B}_\Lambda(\cH)$, then $T$ is unitarily equivalent to $({\bf S}_1\otimes I_\cD,\ldots, {\bf S}_k\otimes I_\cD)$ for some Hilbert space $\cD$ if and only if $T$ is completely non-coisometric  and has characteristic function $\Theta_T=0$.
\end{corollary}

      \bigskip

       %

      \end{document}